\documentclass[review,3p,times]{elsarticle}

\usepackage{lineno,hyperref}
\usepackage{hyperref}
\usepackage{amsfonts}

\usepackage{array,booktabs}
\usepackage{tabularx} 
\usepackage{ragged2e} 

\usepackage{amsmath,amsthm,amsfonts,amssymb,latexsym,epsfig}
\usepackage{hyperref}
\usepackage{color}
\usepackage{graphicx}
\usepackage{epstopdf}  
\usepackage{multirow}
\usepackage{float}
\usepackage{subfigure}
\usepackage{cases}

\usepackage{algorithm}
\usepackage{algorithmic}

\newtheorem{theorem}{Theorem}[section]

\newtheorem{definition}[theorem]{Definition}
\newtheorem{remark}[theorem]{Remark}
\newtheorem{lemma}[theorem]{Lemma}
\newtheorem{corollary}[theorem]{Corollary}
\newtheorem{example}[theorem]{Example}

\biboptions{sort&compress}

\numberwithin{equation}{section}

\modulolinenumbers[5]

\journal{}

\begin{document}

\begin{frontmatter}

\title{Spectral decomposition of $(\star,\epsilon_1,\epsilon_2)$-structured matrix polynomials \tnoteref{mytitlenote}}
\tnotetext[mytitlenote]{Supported by the Research Foundation of Education Department of Hunan Province (Grant No. 23A0266), Hunan Provincial Natural Science Foundation (Grant No. 2025JJ50034) and Natural Science Foundation of Changsha city (Grant No. kq2502074).}


\author[mymainaddress]{Kang Zhao\corref{mycorrespondingauthor}}
\cortext[mycorrespondingauthor]{Corresponding author}
\ead{zkmath@csust.edu.cn }
\address[mymainaddress]{School of Mathematics and Statistics, Changsha University of Science and Technology, Changsha, 410114, China}
\author[mysaddress]{Shifang Yuan}
\author[mymainaddress]{Zhongyun Liu}

\address[mysaddress]{School of Mathematics and Computational Science, Wuyi University, Jiangmen, 529020, China}

\begin{abstract}
We provide the spectral decompositions of $(\star,\epsilon_1,\epsilon_2)$-structured matrix polynomials $P(\lambda)$ in the unified form by a standard pair $(X, J)$ and a parameter matrix $\Gamma$. 
Using the recursive relationship between the coefficient matrices of $P(\lambda)$, equivalent expressions  of these coefficient matrices are provided. When $J$ is assumed to be a block diagonal matrix, we show that the parameter matrix $\Gamma$ has a special structure.\end{abstract}


\end{frontmatter}

\section{Introduction}

Let $\mathbb{K}=\{\mathbb{R}, \mathbb{C}\}$, where $\mathbb{R}$ and $\mathbb{C}$ denote the fields of real and complex numbers, respectively. In this paper, we consider the spectral decomposition of the following $(\star, \epsilon_1, \epsilon_2)$-structured matrix polynomial of $\mathbb{K}^{n\times n}[\lambda]$
\begin{equation}\label{gs-1}
P(\lambda):=\sum_{j=0}^k\lambda^jA_j,
\end{equation} 
with $k\geq 2$,  $A_j\in\mathbb{K}^{n\times n}$, where $\star\in\{H,\ T\}$, $A_j^{\star}=\epsilon_1A_j$ if $j$ is even and $A_j^{\star}=\epsilon_2A_j$ if $j$ is odd, $\epsilon_1, \epsilon_2\in\{1, -1\}$, and $A^H$, $A^T$ denote the conjugate transpose and transpose of a matrix $A$, respectively.  $(\lambda, x)\in\mathbb{K}\times\mathbb{K}^n$ with $x\neq 0$ is said to be an eigenpair of $P(\lambda)$ if 
$$P(\lambda)x=0.$$
For different choices of $\star$, $\epsilon_1$ and $\epsilon_2$, we can get eight classes of linear structured matrix polynomials, namely, $T$-symmetric, $T$-skew-symmetric, $T$-odd, $T$-even, $H$-Hermitian, $H$-skew-Hermitian, $H$-even and $H$-odd \cite{Adhikari-siam-2009}, and the eigenvalue structures of these polynomials are listed in Table \ref{tab1}. These structures, are prototypes of structured matrix polynomials which occur in many applications (see \cite{Mackey-siam-2006-2,Mackey-siam-2006-1,mehrmann-TNA-2002} and the references therein). In recent years, high-order linear control systems attract the attention of researchers, because they are widely used in the fields of mechanical, civil engineering, robotics control theory, and so on. For example, the vibration systems are often described as second-order systems \cite{Lancaster-book-1966,Tisseur-siam-2001}, the three-axis dynamic flight motion simulator systems and the half-car model with advanced suspension systems are often described as third-order systems \cite{Duan-2005,Le-IEEE-2015}, the modelings of the beam and the single-link flexible-joint manipulator are often described as fourth-order systems \cite{NR-1992,JE-book-1991}. 
\begin{table}[t]
\footnotesize
\caption{Eigenvalues of structured matrix polynomials}\label{tab1}
\begin{center}
  \begin{tabular}{|l|l|l|} \hline
   Name & $(\star,\epsilon_1,\epsilon_2)$ & eigenvalues \\ 
   \hline
    $T$-symmetric & $(T, 1, 1)$ & $\lambda$ if $\mathbb{K}=\mathbb{C}$ and $(\lambda, \overline{\lambda})$ if $\mathbb{K}=\mathbb{R}$ \\
    $T$-skew-symmetric & $(T, -1, -1)$ & $\lambda$ if $\mathbb{K}=\mathbb{C}$ and $(\lambda, \overline{\lambda})$ if $\mathbb{K}=\mathbb{R}$ \\ 
    $T$-even & $(T, 1, -1)$ & $(\lambda, -\lambda)$ if $\mathbb{K}=\mathbb{C}$ and $(\lambda, \overline{\lambda}, -\lambda, -\overline{\lambda})$ if $\mathbb{K}=\mathbb{R}$ \\ 
    $T$-odd & $(T, -1, 1)$ & $(\lambda, -\lambda)$ if $\mathbb{K}=\mathbb{C}$ and $(\lambda, \overline{\lambda}, -\lambda, -\overline{\lambda})$ if $\mathbb{K}=\mathbb{R}$ \\
    $H$-Hermitian & $(H, 1, 1)$ & $(\lambda, \overline{\lambda})$\\
    $H$-skew-Hermitian & $(H, -1, -1)$ & $(\lambda, \overline{\lambda})$\\
    $H$-even & $(H, 1, -1)$ & $(\lambda, -\overline{\lambda})$\\
    $H$-odd & $(H, -1, 1)$ & $(\lambda, -\overline{\lambda})$\\
     \hline
  \end{tabular}
\end{center}
\end{table}

For general matrix polynomials of arbitrary degree, Gohberg, Lancaster and Rodman \cite{Lancaster-book} provided a powerful GLR theory, which can characterize the general matrix polynomials by a standard pair $(X, T)$.  However,  they did not consider the structures of the coefficient matrices, such as symmetry, skew-symmetry, Hermitian, skew-Hermitian and so on, and they also can not give the parametric formula of all the coefficient matrices. For the last decade, the real-valued spectral decompositions of second-order structured-matrix polynomials have been intensively investigated. With the coefficient matrix $M$ being nonsingular, Chu and Xu \cite{Chu-MC-2009} derived a real-valued spectral decomposition of the symmetric system $Q(\lambda)=\lambda^2M+\lambda C+K$, where $M, C, K\in\mathbb{R}^{n\times n}$ are all symmetric. They provided that a matrix pair $(X,\ \mathfrak{J})$ is a real standard pair of $Q(\lambda)$ if and only if there exists a real nonsingular symmetric matrix $S$ such that $\mathfrak{J}S=S\mathfrak{J}^T$ and $XSX^T=0$, and $(M, C, K)$ can be characterized by $(X,\ \mathfrak{J})$ and $S$. Moreover, they also derived the structure of $S$, which is a symmetric block diagonal matrix with its blocks or subblocks being partitioned into upper triangular Hankel blocks. When all eigenvalues of $Q(\lambda)$ are semi-simple, they point out that there is a real standard pair such that $S$ has a simpler structure.  Inspired by \cite{Chu-MC-2009}, Jia and Wei \cite{Jia-siam-2011} derived a real-valued spectral decomposition of the damped gyroscopic system $G(\lambda)=\lambda^2M+\lambda D+K$ in terms of a real standard pair $(X,\, T)$ and the parameter matrix $S$, where $M, K\in\mathbb{R}^{n\times n}$ are symmetric and $D\in\mathbb{R}^{n\times n}$ is skew-symmetric. The matrix $S$ is block diagonal and skew-symmetric, which satisfies $TS=-ST^T$ and $XSX^T=0$, and its blocks are of upper triangular skew-Hankel forms. With the matrices $M$ and $K$ being nonsingular, Zhao \cite{Zhao-jsv-2024} gave a new proof of the real-valued spectral decomposition of the damped symmetric system $Q(\lambda)$. They provided a sufficient and necessary condition that the quadratic system can preserve no spill-over, and characterized a new set of solutions for the partial eigenvalue assignment problem. Based on the real-valued spectral decomposition of a non-symmetric structured-matrix polynomial, Qian \cite{Qian-siam-2017} and Zhao \cite{Zhao-siam-2023} considered the eigenvalue embedding problems of the undamped and damped vibroacoustic systems, respectively, which can preserve the block structures of the coefficient matrices, the symmetry properties of the system matrices and the no spill-over property, simultaneously.  Recently, Zhao \cite{Zhao-mssp-2025} provided the real-valued spectral decomposition of the $T$-symmetric matrix polynomial with arbitrary degree.  Therefore, it is very interesting and challenging to derive the spectral decomposition of all the $(\star, \epsilon_1, \epsilon_2)$-structured matrix polynomials as in (\ref{gs-1}). 

Motivated by \cite{Chu-MC-2009,Jia-siam-2011,Zhao-mssp-2025}, in this paper we derive a spectral decomposition of the $(\star,\epsilon_1,\epsilon_2)$-structured matrix polynomials in terms of a standard pair $(X,\ J)$ and a parameter matrix $\Gamma$, which satisfy $\Gamma^{\star}=\epsilon_2\Gamma$, $\Gamma J^{\star}=\epsilon_1\epsilon_2J\Gamma$, and $XJ^r\Gamma X^{\star}=0$ for $r=0,1,\ldots,k-2$. For some special choices of $\star$, $\epsilon_1$ and $\epsilon_2$, the spectral decompositions given by \cite{Chu-MC-2009,Jia-siam-2011,Zhao-mssp-2025} can be seen as special cases of our results. Moreover, the structure of $\Gamma$ here has a more complicated formula than those for the damped symmetric system in \cite{Chu-MC-2009} and undamped gyroscopic system in \cite{Jia-siam-2011}. And the formula of $\Gamma$ cannot be obtained using the method of \cite{Chu-MC-2009,Jia-siam-2011}, since the blocks of $\Gamma$ are symmetric skew-Hamiltonian and skew-symmetric skew-Hamiltonian for some $(\star,\epsilon_1,\epsilon_2)$-structured matrix polynomials. The main contributions of this paper are: 

(i) The spectral decompositions of all the $(\star, \epsilon_1, \epsilon_2)$-structured matrix polynomials are provided in the unified form. And all the $(\star, \epsilon_1, \epsilon_2)$-structured coefficient matrices  are characterized by a standard pair $(J,\ X)$ and a parameter matrix $\Gamma$. 

(ii) According to the eigenstructure of the $(\star, \epsilon_1, \epsilon_2)$-structured matrix polynomial $P(\lambda)$ listed in Table \ref{tab1}, the structures of the parameter matrix $\Gamma$ are characterized. When all eigenvalues of $P(\lambda)$ are semi-simple, $\Gamma$ can be specified in the simplest structure.

Throughout this paper, the following notations will be used. Denoted by $\mathbb{C}^{m\times n}$ and $\mathbb{R}^{m\times n}$ the set of all complex and real $m\times n$ matrices, respectively, $\mathbb{N}^+$ the set of all positive integers, ${\rm i}\mathbb{R}$ the set of all pure imaginary numbers and $I_m$ the identity matrix.

\section{The spectral decomposition}
In this section, we will provide the spectral decompositions of regular and singular $(\star,\epsilon)$-structured matrix polynomials, respectively.  The notion of \emph{standard pair} introduced in \cite{Lancaster-book} will be a fundamental tool for our development. The canonical form of a standard pair can be considered as a generalization of the Jordan canonical form to matrix polynomials.

\begin{definition}\label{lem-1}\cite[Chapter 7]{Lancaster-book}
A pair of matrices $(X,\ J)\in\mathbb{K}^{n\times kn}\times\mathbb{R}^{kn\times kn}$ is called a standard pair for the matrix polynomial $P(\lambda)$ as in (\ref{gs-1}) if and only if the matrix 
\begin{equation}\label{xgs-1}X_L=\begin{bmatrix}
X\\
XJ\\
\vdots\\
XJ^{k-1}\\
\end{bmatrix}\in\mathbb{K}^{kn\times kn}.
\end{equation}
is nonsingular and the equation
\begin{equation}\label{gs-2}
A_kXJ^k+A_{k-1}XJ^{k-1}+\cdots+A_1XJ+A_0X=0
\end{equation}
holds.
\end{definition}

\subsection{Regular matrix polynomial}

Let $\epsilon:=\epsilon_1\epsilon_2$. For any matrix $\Phi\in\mathbb{K}^{m\times m}$, let
\begin{equation}\label{gs-5}
\mathcal{S}_{(\Phi,\star,\epsilon_1,\epsilon_2)}=\left\{S\in\mathbb{K}^{m\times m}|S^{\star}=\epsilon_2S,\ S\Phi^{\star}=\epsilon\Phi S\right\}.
\end{equation}

It is well known that $P(\lambda)$ is regular if the coefficient matrix $A_k$ is nonsingular, and then $P(\lambda)$ has $kn$ finite eigenvalues.
In this subsection, we always assume that the coefficient matrix $A_k$ of $P(\lambda)$  is nonsingular. 
As is known, linearization is a rather convenient way to cast a matrix polynomial as a linear problem \cite{Mackey-siam-2006-1}. It is easy to verify that the equation (\ref{gs-2}) can be equivalently rewritten as 
\begin{equation}\label{gs-3}
\mathcal{A}X_LJ-\epsilon\mathcal{B}X_L=0,
\end{equation}
where $\mathcal{A}, \mathcal{B}\in\mathbb{K}^{kn\times kn}$, if $k$ is even,
$$\scriptsize{
\mathcal{A}=\begin{bmatrix}
A_1 & A_2 & A_3  & \cdots & A_{k-1} & A_k\\
\epsilon A_2 & \epsilon A_3 & \epsilon A_4  & \cdots & \epsilon A_k & 0\\
A_3 & A_4 & A_5  & \cdots & 0 & 0\\
 \vdots & \vdots & \vdots & \ddots & \vdots & \vdots\\
A_{k-1} & A_k & 0  & \cdots & 0 & 0\\
\epsilon A_k & 0 & 0  & \cdots & 0 & 0\\
\end{bmatrix}, \mathcal{B}=\begin{bmatrix}
-\epsilon A_0 & 0 & 0 & \cdots & 0 & 0\\
0 & A_2 & A_3  & \cdots & A_{k-1} & A_k\\
0 & \epsilon A_3 & \epsilon A_4  & \cdots & \epsilon A_k & 0\\
0 & A_4 & A_5  & \cdots & 0 & 0\\
 \vdots & \vdots & \vdots & \ddots & \vdots & \vdots\\
0 & A_k & 0  & \cdots & 0 & 0\\
\end{bmatrix},}
$$
and if $k$ is odd,
$$\scriptsize{
\mathcal{A}=\begin{bmatrix}
A_1 & A_2 & A_3  & \cdots & A_{k-1} & A_k\\
\epsilon A_2 & \epsilon A_3 & \epsilon A_4  & \cdots & \epsilon A_k & 0\\
A_3 & A_4 & A_5  & \cdots & 0 & 0\\
 \vdots & \vdots & \vdots & \ddots & \vdots & \vdots\\
\epsilon A_{k-1} & \epsilon A_k & 0  & \cdots & 0 & 0\\
 A_k & 0 & 0  & \cdots & 0 & 0\\
\end{bmatrix}, \mathcal{B}=\begin{bmatrix}
-\epsilon A_0 & 0 & 0 & \cdots & 0 & 0\\
0 & A_2 & A_3  & \cdots & A_{k-1} & A_k\\
0 & \epsilon A_3 & \epsilon A_4  & \cdots & \epsilon A_k & 0\\
0 & A_4 & A_5  & \cdots & 0 & 0\\
 \vdots & \vdots & \vdots & \ddots & \vdots & \vdots\\
0 & \epsilon A_k & 0  & \cdots & 0 & 0\\
\end{bmatrix}.}
$$
It follows from the definition of $P(\lambda)$ in (\ref{gs-1}) that $\mathcal{A}^{\star}=\epsilon_2\mathcal{A}$ and $\mathcal{B}^{\star}=\epsilon_1\mathcal{B}$. 
Clearly, the matrix $\mathcal{A}$ defined in  (\ref{gs-3}) is nonsingular, since $A_k$ is nonsingular. 
Let
\begin{equation}\label{gs-4}
\Gamma:=\left(X_L^{\star}\mathcal{A}X_L\right)^{-1}\in\mathbb{K}^{kn\times kn}.
\end{equation}
Define 
\begin{equation}\label{g-t}
\tau(j)=\left\{\begin{array}{ll}
1, &\text{if}\ \ j\ \ \text{is}\ \ \text{odd},\\
\epsilon, &\text{if}\ \ j\ \ \text{is}\ \ \text{even}.\\
\end{array}\right.
\end{equation}
\begin{theorem}\label{thm-1}
Let $(X,\ J)$ be a pair of matrices, where $X\in\mathbb{R}^{n\times kn}$ and $J\in\mathbb{R}^{kn\times kn}$. Then $(X,\ J)$  is a standard pair of $P(\lambda)$ as  in (\ref{gs-1}) if and only if the matrix $X_L$ defined in (\ref{xgs-1}) and $XJ^{k-1}\Gamma X^{\star}$ are nonsingular, and there exists an $kn\times kn$ nonsingular matrix $\Gamma\in\mathcal{S}_{(J,\star,\epsilon_1,\epsilon_2)}$ satisfying
\begin{equation}\label{gs-6}
XJ^r\Gamma X^{\star}=0, \ r=0, 1, \ldots, k-2.
\end{equation}
And in this case, the coefficient matrices $A_0, A_1, \ldots, A_k$ of $P(\lambda)$ can be expressed as 
\begin{equation}\label{gs-7}
\left\{\begin{array}{l}
A_k=(XJ^{k-1}\Gamma X^{\star})^{-1},\\
A_j=-\sum_{i=j+1}^kA_iXJ^{k-j+i-1}\Gamma X^{\star}A_k,\ j=0, 1, \ldots, k-1.
\end{array}\right.
\end{equation}
\end{theorem}
\begin{proof}
We only give the proof of the case that $k$ is odd, and the case that $k$ is even can be proved similarly. 

(Necessity) If $(X,\ J)$ is a standard pair of $P(\lambda)$, we can see from Definition \ref{lem-1} that the matrix $X_L$ defined in (\ref{xgs-1}) is nonsingular and the equation (\ref{gs-2}) holds, which implies that the matrix $\Gamma$ in (\ref{gs-4}) is well-defined.  Multiplying on the left side of (\ref{gs-2}) by $X^{\star}$ leads to 
\begin{equation}\label{gs-8}
X^{\star}A_0X=-\sum_{j=1}^kX^{\star}A_jXJ^j.
\end{equation}
Taking the $\star$-transpose of (\ref{gs-2}) and multiplying on the right side by $X$, we get
$$\epsilon_2(J^k)^{\star}X^{\star}A_kX+\epsilon_1(J^{k-1})^{\star}X^{\star}A_{k-1}X+\cdots+\epsilon_2J^{\star}X^{\star}A_1X+\epsilon_1X^{\star}A_0X=0,$$
which implies that
\begin{equation}\label{gs-9}
X^{\star}A_0X=-\epsilon\sum_{j=1}^k\tau(j)(J^j)^{\star}X^{\star}A_jX.
\end{equation}
By the definition of $\Gamma$ in (\ref{gs-4}), we have
$$\Gamma^{-1}=X_L^{\star}\mathcal{A}X_L=\sum_{s=1}^k\tau(s)\left[\sum_{j=0}^{k-s}\tau(j+s)(J^j)^{\star}X^{\star}A_{j+s}X\right]J^{s-1}.$$
Then, it follows from (\ref{gs-8}) and (\ref{gs-9}) that 
$${\begin{array}{l}
\Gamma^{-1}J \\
 =\sum\limits_{s=1}^k\tau(s)\left[\sum\limits_{j=0}^{k-s}\tau(j+s)(J^j)^{\star}X^{\star}A_{j+s}X\right]J^s\\
 =\sum\limits_{s=1}^kX^{\star}A_sXJ^s+\sum\limits_{s=1}^{k-1}\tau(s)\left[\sum\limits_{j=1}^{k-s}\tau(j+s)(J^j)^{\star}X^{\star}A_{j+s}X\right]J^s\\
 =-X^{\star}A_0X+\sum\limits_{s=1}^{k-1}\tau(s)\left[\sum\limits_{j=1}^{k-s}\tau(j+s)(J^j)^{\star}X^{\star}A_{j+s}X\right]J^s\\
 =\epsilon\sum\limits_{j=1}^k\tau(j)(J^j)^{\star}X^{\star}A_jX+\sum\limits_{s=1}^{k-1}\tau(s)\left[\sum\limits_{j=1}^{k-s}\tau(j+s)(J^j)^{\star}X^{\star}A_{j+s}X\right]J^s\\
 =\epsilon J^{\star}\left(\sum\limits_{j=1}^k\tau(j)(J^{j-1})^{\star}X^{\star}A_jX+\sum\limits_{s=1}^{k-1}\tau(s+1)\left[\sum\limits_{j=1}^{k-s}\tau(j+s)(J^{j-1})^{\star}X^{\star}A_{j+s}X\right]J^s\right)\\
  =\epsilon J^{\star}\left(\sum\limits_{j=1}^k\tau(j)(J^{j-1})^{\star}X^{\star}A_jX+\sum\limits_{s=2}^{k}\tau(s)\left[\sum\limits_{j=1}^{k-s+1}\tau(j+s-1)(J^{j-1})^{\star}X^{\star}A_{j+s-1}X\right]J^{s-1}\right)\\
=\epsilon J^{\star}\sum\limits_{s=1}^{k}\tau(s)\left[\sum\limits_{j=1}^{k-s+1}\tau(j+s-1)(J^{j-1})^{\star}X^{\star}A_{j+s-1}X\right]J^{s-1}\\
=\epsilon J^{\star}\sum\limits_{s=1}^{k}\tau(s)\left[\sum\limits_{j=0}^{k-s}\tau(j+s)(J^{j})^{\star}X^{\star}A_{j+s}X\right]J^{s-1}\\
=\epsilon J^{\star}\Gamma^{-1},\\
\end{array}}$$
which implies that $J\Gamma=\epsilon\Gamma J^{\star}$. Since $\mathcal{A}^{\star}=\epsilon_2\mathcal{A}$, it is easy to see from (\ref{gs-4}) that $\Gamma^{\star}=\epsilon_2\Gamma$, i.e., $\Gamma\in\mathcal{S}_{(J,\star,\epsilon_1,\epsilon_2)}$. Note that $\mathcal{A}=(X_L\Gamma X_L^{\star})^{-1}$, then we have 
\begin{equation}\label{gs-10}
X_L\Gamma X_L^{\star}\mathcal{A}=I_{kn}.
\end{equation}
It follows from the last column block of (\ref{gs-10}) that 
$A_k=(XJ^{k-1}\Gamma X^{\star})^{-1}$ and 
$$X\Gamma X^{\star}=0,\ XJ\Gamma X^{\star}=0,  XJ^2\Gamma X^{\star}=0,\ \ldots,\ XJ^{k-2}\Gamma X^{\star}=0,$$
i.e,. (\ref{gs-6}) holds. Post-multiplying (\ref{gs-2}) by $\Gamma X^{\star}$ and applying (\ref{gs-6}), we can obtain that
$$A_kXJ^k\Gamma X^{\star}+A_{k-1}XJ^{k-1}\Gamma X^{\star}=0.$$
Then, $A_{k-1}=-A_kXJ^k\Gamma X^{\star}A_k.$ Similarly, post-multiplying (\ref{gs-2}) by $J^{k-j-1}\Gamma X^{\star}$, we can get
$$A_j=-\sum_{i=j+1}^kA_iXJ^{k-j+i-1}\Gamma X^{\star}A_k,\ \ j=0, 1, \ldots, k-2,$$
i.e., all the coefficient matrices of $P(\lambda)$ in (\ref{gs-1}) can be expressed as in (\ref{gs-7}).

(Sufficiency) Suppose that there exists a nonsingular matrix $\Gamma\in\mathcal{S}_{(J,\star,\epsilon_1,\epsilon_2)}$ such that (\ref{gs-6}) holds, then we have 
\begin{equation}\label{gs-11}
X_L\Gamma X_L^{\star}\\
=\begin{bmatrix}
0 & 0 & \cdots & 0 & \epsilon^{k-1}XJ^{k-1}\Gamma X^{\star}\\
0 & 0 & \cdots & \epsilon^{k-2}XJ^{k-1}\Gamma X^{\star} & \epsilon^{k-1}XJ^k\Gamma X^{\star}\\
\vdots & \vdots & \ddots & \vdots & \vdots\\
0 & \epsilon XJ^{k-1}\Gamma X^{\star} & \cdots & \epsilon^{k-2}XJ^{2k-4}\Gamma X^{\star} & \epsilon^{k-1}XJ^{2k-3}\Gamma X^{\star}\\
XJ^{k-1}\Gamma X^{\star} & \epsilon XJ^k\Gamma X^{\star} & \cdots & \epsilon^{k-2}XJ^{2k-3}\Gamma X^{\star} & \epsilon^{k-1}XJ^{2k-2}\Gamma X^{\star}\\
\end{bmatrix}
\end{equation}
Note that the matrices $X_L$ and $\Gamma$ are nonsingular, it follows that $X_L\Gamma X_L^{\star}$ is nonsingular, i.e., $XJ^{k-1}\Gamma X^{\star}$ is also nonsingular. Since $J\Gamma =\epsilon\Gamma J^{\star}$ and $\Gamma^{\star}=\epsilon_2\Gamma$, it is easy to verify that 
\begin{equation}\label{gs-12}
J^s\Gamma=\epsilon^s\Gamma(J^s)^{\star}, \ \ s\in\mathbb{Z}^+.
\end{equation}
Since $k$ is odd and $\Gamma^{\star}=\epsilon_2\Gamma$, we can see from (\ref{gs-7}) and (\ref{gs-12}) that 
$$A_k^{\star}=\epsilon_2\left(X\Gamma(J^{k-1})^{\star}X^{\star}\right)^{-1}=\epsilon_2\epsilon^{k-1}(XJ^{k-1}\Gamma X^{\star})^{-1}=\epsilon_2A_k,$$
which implies that the matrix $A_k$ given by (\ref{gs-7}) is well-defined. Since $k$ is odd, we have
$$A_{k-1}^{\star}=-(A_kXJ^k\Gamma X^{\star}A_k)^{\star}=-A_kX\Gamma^{\star}(J^k)^{\star}X^{\star}A_k=\epsilon_2\epsilon^kA_{k-1}=\epsilon_1A_{k-1},$$
i.e., the matrix $A_{k-1}$ given by (\ref{gs-7}) is well-defined. Moreover, we can obtain from (\ref{gs-7}) and (\ref{gs-12}) that
$$\begin{array}{ll}
A_{k-2}^{\star}&=-(A_kXJ^{k+1}\Gamma X^{\star}A_k)^{\star}-(A_{k-1}XJ^k\Gamma X^{\star}A_k)^{\star}\\
&=-\epsilon_2\epsilon^{k+1}A_kXJ^{k+1}\Gamma X^{\star}A_k-\epsilon_1\epsilon^kA_kXJ^k\Gamma X^{\star}A_{k-1}\\
&=-\epsilon_2\left(A_kXJ^{k+1}\Gamma X^{\star}A_k+A_kXJ^k\Gamma X^{\star}A_kXJ^k\Gamma X^{\star}A_k\right)\\
&=-\epsilon_2\left(A_kXJ^{k+1}\Gamma X^{\star}A_k+A_{k-1}XJ^k\Gamma X^{\star}A_k\right)\\
&=\epsilon_2A_{k-2}.
\end{array}
$$
Similarly, we can prove that all the matrices given by (\ref{gs-7}) satisfying 
$A_j^{\star}=\epsilon_1A_j$ when $j$ is even and $A_j^{\star}=\epsilon_2A_j$ when $j$ is odd, i.e., the matrices $A_j$ $(j=0, 1, \ldots, k)$ given by (\ref{gs-7}) are all well-defined. 

Substituting (\ref{gs-7}) into (\ref{gs-11}), we can get $X_L\Gamma X_L^{\star}=\mathcal{A}^{-1}$ by simple matrix multiplication, which implies that $\Gamma$ has the form given by (\ref{gs-4}) and $X_L^{-1}=\Gamma X_L^{\star}\mathcal{A}$. Thus, we can obtain that
\begin{equation}\label{gs-13}
\begin{array}{l}
A_kXJ^kX_L^{-1}\\
=A_kXJ^k\Gamma\left[X^{\star}, J^{\star}X^{\star},\ldots, (J^{k-1})^{\star}X^{\star}\right]\mathcal{A}\\
=A_kXJ^k\left[\Gamma X^{\star}, \epsilon J\Gamma X^{\star},\ldots, \epsilon^{k-1}J^{k-1}\Gamma X^{\star}\right]\mathcal{A}\\
=A_kXJ^k\left[\tau(1)\sum\limits_{j=0}^{k-1}J^j\Gamma X^{\star}\tau(j+1)\epsilon^{j}A_{j+1},\tau(2)\sum\limits_{j=0}^{k-2}J^j\Gamma X^{\star}\tau(j+2)\epsilon^{j}A_{j+2}, \ldots,\right.\\
\left.\ \ \ \ \ \ \ \ \ldots,\tau(k-1)\sum\limits_{j=0}^{1}J^j\Gamma X^{\star}\tau(j+k-1)\epsilon^{j}A_{j+k-1}, \tau(k)\Gamma X^{\star}A_k\right].\\ 
\end{array}
\end{equation} 
For any numbers $s, j$ with $s=1, 2, \ldots, k$ and $j=0,1,\ldots, k-1$, we always have 
\begin{equation}\label{gs-14}
\tau(s)\tau(j+s)\epsilon^j=1.
\end{equation} 
In fact, if $s$ and $j$ are all odd, then we have $\tau(s)=1$, $\tau(j+s)=\epsilon$ and $\epsilon^j=\epsilon$, which implies that $\tau(s)\tau(j+s)\epsilon^j=1$. The other cases can be proved similarly. 
Substituting (\ref{gs-7}) and (\ref{gs-14}) into (\ref{gs-13}) leads to
$$\begin{array}{l}
A_kXJ^kX_L^{-1}\\
=A_kXJ^k\left[\sum\limits_{j=0}^{k-1}J^j\Gamma X^{\star}A_{j+1},\sum\limits_{j=0}^{k-2}J^j\Gamma X^{\star}A_{j+2}, \ldots,\sum\limits_{j=0}^{1}J^j\Gamma X^{\star}A_{j+k-1}, \Gamma X^{\star}A_k\right].\\ 
=\left[-\epsilon_1A_0^{\star}, -\epsilon_2A_1^{\star}, \ldots, -\epsilon_2A_{k-2}^{\star}, -\epsilon_{1}A_{k-1}^{\star}\right]\\
=\left[-A_0, -A_1, \ldots, -A_{k-1}\right],\\
\end{array}
$$
which implies that
$$A_kXJ^k=\left[-A_0, -A_1, \ldots, -A_{k-1}\right]X_L,$$
i.e., 
$$A_kXJ^k+A_{k-1}XJ^{k-1}+\cdots+A_1XJ+A_0X=0.$$
This completes the proof.
\end{proof}
\begin{remark}(Recovery of results in Zhao \cite{Zhao-mssp-2025})
Taking $\mathbb{K}=\mathbb{R}$ and $\star=T$, then the real-valued spectral decomposition of the high-order symmetric matrix polynomial provided in \cite{Zhao-mssp-2025} coincides with the Theorem \ref{thm-1} by setting $\epsilon_1=1$ and $\epsilon_2=1$.
\end{remark}

We can see from (\ref{gs-7}) that each term of the formula $A_j$ ($j=0, 1,\ldots, k-1$) contains the matrix pair $(X,\ J)$, which may bring difficulties in applications, such as inverse eigenvalue problem, eigenvalue embedding problem and so on. Using the recursive relationship between the coefficient matrices $A_j$, we give an equivalent expression of (\ref{gs-7}). And there is only one term of the new expression of $A_j$, which contains the matrix pair $(X,\ J)$.  

Let $t\in\mathbb{Z}^+ (1\leq t<k)$ and $\gamma_1,\ldots,\gamma_{t+1}\in\mathbb{Z}^+$ satisfy $1\leq\gamma_1\leq \cdots\leq \gamma_{t+1}<k$. Let $\mathcal{L}(\gamma_1,\ldots,\gamma_{t+1})$ ( denoted simply by $\mathcal{L}(t+1)$) be the set of all permutations of the numbers $\gamma_1,\ldots,\gamma_{t+1}$. In fact, if all the distinct numbers of $\gamma_1,\ldots,\gamma_{t+1}$ are $\eta_1,\ldots,\eta_r$, and the number of $\eta_j$ is $t_j$, $j=1,\ldots,r$. Clearly, $t_1+t_2+\cdots+t_r=t+1$. It is well known that the number of all the permutations in $\mathcal{L}(t+1)$ is $\frac{(t+1)!}{\Pi_{j=1}^r(t_j!)}$, where $t_j!=1\times 2\times \cdots\times t_j$. For any non-negative number $d\in \mathbb{Z}$, define 
\begin{equation}\label{gs-15}
\mathcal{I}_{(t,\ d)}:=\left\{(\gamma_1,\ldots,\gamma_{t+1})\in\mathcal{L}(t+1)\ |\ \sum_{j=1}^{t+1}\gamma_j=tk+d\right\}.
\end{equation}
\begin{example}
Let $k=7$, $t=3$ and $d=1$, then
$$\begin{array}{ll}
\mathcal{I}_{(3, 1)}=&\{(4,6,6,6),(6,4,6,6),(6,6,4,6),(6,6,6,4), (5,5,6,6)\\
& \ \,(5,6,5,6),(6,5,6,5),(6,6,5,5),(5,6,6,5),(6,5,5,6)\}.\\
\end{array}
$$
\end{example}
\begin{lemma}\label{lem-2}
Let $s, m\in\mathbb{Z}^+$ satisfy  $3\leq s\leq k+1$ and $1\leq m\leq s-2$. Then, $\mathcal{I}_{(m+1, k-s)}=\mathcal{N}$, where
$$\mathcal{N}=\bigcup_{l=1}^{s-2}\bigcup_{(\gamma_1,\ldots,\gamma_{m+1})\in\mathcal{I}_{(m,k-s+l)}}\mathcal{L}(k-l, \gamma_1,\ldots, \gamma_{m+1}).$$ 
\end{lemma}
\begin{proof}
For any $(\gamma_1,\ldots,\gamma_{m+1})\in\mathcal{I}_{(m,k-s+l)}$, we can see from (\ref{gs-15}) that 
$$(k-l)+\gamma_1+\cdots+\gamma_{m+1}=(m+1)k+(k-s),$$
for $m=1, 2, \ldots, s-2$ and $l=1,2,\ldots,s-2$, which implies that $$(k-l, \gamma_1,\ldots, \gamma_{m+1})\in\mathcal{I}_{(m+1, k-s)},$$ i.e., $\mathcal{N}\subseteq\mathcal{I}_{(m+1,k-s)}.$
Consider the following equation
\begin{equation}\label{gs-17}
x_1+x_2+\cdots+x_{m+2}=(m+2)k-s,
\end{equation}
where $1\leq x_j<k, j=1,\ldots, m+2$. From (\ref{gs-15}), it is easy to see that the solution set of (\ref{gs-17}) is $\mathcal{I}_{(m+1,k-s)}.$
Now, we prove that $\mathcal{I}_{(m+1,k-s)}\subseteq \mathcal{N}$. Since $x_j<k, j=1,\ldots, m+2$, it is easy to verify from (\ref{gs-17}) that the smallest choice of $x_j$ can not less than $k-s+1+\frac{m(m+1)}{2}$. In fact, if $x_1=k-s+\frac{m(m+1)}{2}$, then we can take $$(x_2,x_3,\ldots,x_{m+1},x_{m+2})=(k-m,k-m+1,\ldots,k-1,k),$$
which implies that (\ref{gs-17}) holds. This is a contradiction with $x_{m+2}<k$.  

Set $x_{m+2}=k-l$, $l=1,\ldots,s-2$. This choice is reasonable, since $\frac{m(m+1)}{2}\geq 1$. Then, 
\begin{equation}\label{gs-18}
x_1+\cdots+x_{m+1}=mk+k-s+l.
\end{equation}
Clearly, the solution set of (\ref{gs-18}) is $\mathcal{I}_{(m, k-s+l)}$, which implies that the solution set of (\ref{gs-17}) with $x_{m+2}=k-l$ is 
 $$\bigcup_{(x_1,\ldots,x_{m+1})\in\mathcal{I}_{(m,k-s+l)}}\mathcal{L}(x_1,\ldots, x_{m+1}, k-l).$$
And then the solution set of (\ref{gs-17}) is 
 $$\bigcup_{l=1}^{s-2}\bigcup_{(x_1,\ldots,x_{m+1})\in\mathcal{I}_{(m,k-s+l)}}\mathcal{L}(x_1,\ldots, x_{m+1},k-l),$$
i.e., $\mathcal{I}_{(m+1,k-s)}\subseteq\mathcal{N}$.
\end{proof}

The following result can be proved similarly as the Corollary 2.4 in \cite{Zhao-mssp-2025}. However, they only considered the case of $T$-symmetric matrix polynomial and  did not provide a detail proof of the recursive steps. Based on the Lemma \ref{lem-2}, we provide a complete proof for the $(\star, \epsilon_1, \epsilon_2)$-structured matrix polynomials. 
\begin{corollary}\label{cor-1}
The matrices $A_0, A_1, \ldots, A_{k}$ given by (\ref{gs-7}) in Theorem \ref{thm-1} can be equivalently rewritten as
\begin{equation}\label{gs-16}
\left\{\begin{array}{l}
A_k=(XJ^{k-1}\Gamma X^{\star})^{-1},\\
A_{k-1}=-A_kXJ^k\Gamma  X^{\star}A_k,\\
A_j=-A_kXJ^{2k-j-1}\Gamma X^{\star}A_k+g_j, \ \ j=0, 1, \ldots, k-2,\\
\end{array}\right.
\end{equation}
where 
\begin{align}
&g_j=\sum_{t=1}^{k-j-1}(-1)^{t-1}g_j^{(t)},\label{gs-g-1}\\
&g_j^{(t)}=\sum_{(\gamma_1,\ldots,\gamma_{t+1})\in\mathcal{I}_{(t,j)}}(A_{\gamma_1}A_k^{-1}A_{\gamma_2}A_{k}^{-1}\cdots A_{k}^{-1}A_{\gamma_t}A_{k}^{-1}A_{\gamma_{t+1}}).\label{gs-g-2}
\end{align}
\end{corollary}
\begin{proof}
Since $A_k$ is nonsingular, the matrix $A_{k-1}$ given by (\ref{gs-16}) can be rewritten as 
\begin{equation}\label{gs-19}
XJ^k\Gamma X^{\star}=-A_k^{-1}A_{k-1}A_k^{-1}.
\end{equation}
By Theorem \ref{thm-1}, we know that 
\begin{equation}\label{gs-20}
A_{k-2}=-A_kXJ^{k+1}\Gamma X^{\star}A_k-A_{k-1}XJ^k\Gamma X^{\star}A_k.
\end{equation}
 Substituting (\ref{gs-19})  into (\ref{gs-20}) leads to 
 \begin{equation}\label{gs-21}
 A_{k-2}=-A_kXJ^{k+1}\Gamma X^{\star}A_k+A_{k-1}A_k^{-1}A_{k-1}.
 \end{equation}
Clearly, $\mathcal{I}_{(1,k-2)}=\{(k-1, k-1)\}$, which implies that (\ref{gs-21}) can be rewritten as $A_{k-2}=-A_kXJ^{k+1}\Gamma X^{\star}A_k+g_{k-2}$, where $g_{k-2}$ is defined by (\ref{gs-g-1}). Suppose that the matrices $A_{k-3}, A_{k-4}, \ldots, A_{k-s+1}$ can be rewritten as in (\ref{gs-16}), where $3\leq s\leq k+1$. Then, 
\begin{equation}\label{gs-22} 
\left\{\begin{array}{ll}
XJ^k\Gamma X^{\star} & =-A_k^{-1}A_{k-1}A_k^{-1}\\
XJ^{k+1}\Gamma X^{\star} & =-A_k^{-1}A_{k-2}A_k^{-1}+A_k^{-1}g_{k-2}A_k^{-1},\\
& \vdots\\
XJ^{k+s-2}\Gamma X^{\star} & =-A_{k}^{-1}A_{k-s+1}A_k^{-1}+A_k^{-1}g_{k-s+1}A_k^{-1}.\\
\end{array}\right.
\end{equation}
Now, we prove that $A_{k-s}$ can also be rewritten as in (\ref{gs-16}). By Theorem \ref{thm-1}, we have 
\begin{equation}\label{gs-23}
A_{k-s}=-\sum_{i=k-s+1}^kA_iXJ^{s+i-1}\Gamma X^{\star}A_k.
\end{equation}
Substituting (\ref{gs-22}) into (\ref{gs-23}), we can obtain that 
\begin{equation}\label{gs-24}
\begin{array}{ll}
A_{k-s}&=-A_kXJ^{k+s-1}\Gamma X^{\star}A_k+\sum_{i=1}^{s-1}A_{k-i}A_k^{-1}A_{k-s+i}\\
&\ \ \ -\sum_{l=1}^{s-2}\left[\sum_{t=1}^{s-l-1}(-1)^{t-1}A_{k-l}A_k^{-1}g_{k-s+l}^{(t)}\right].\\
\end{array}
\end{equation}
By the definition in (\ref{gs-15}), it is easy to see that $\mathcal{I}_{(1,k-s)}=\bigcup\limits_{i=1}^{s-1}\{(k-i, k-s+i)\}.$ Then, we have 
\begin{equation}\label{gs-25}
\sum\limits_{i=1}^{s-1}A_{k-i}A_k^{-1}A_{k-s+i}=\sum_{(\gamma_1,\gamma_2)\in\mathcal{I}_{(1,k-s)}}A_{\gamma_1}A_k^{-1}A_{\gamma_2}=g_{k-s}^{(1)}.
\end{equation}
Next, we consider the third term of $A_{k-s}$. Set $t=m$, where $1\leq m\leq s-2$, then the third term of $A_{k-s}$ is 
$$(-1)^{m}\sum_{l=1}^{s-2}A_{k-l}A_k^{-1}g_{k-s+l}^{(m)}.$$
For any $(\gamma_1,\ldots,\gamma_{m+1})\in\mathcal{I}_{(m,k-s+l)}$, we have 
$$(k-l)+\gamma_1+\cdots+\gamma_{m+1}=(m+1)k+(k-s),$$
where $1\leq l\leq s-2$. It follows from Lemma \ref{lem-2} that 
$$\mathcal{I}_{(m+1,k-s)}=\bigcup_{l=1}^{s-2}\bigcup_{(\gamma_1,\ldots,\gamma_{m+1})\in\mathcal{I}_{(m,k-s+l)}}\mathcal{L}(\gamma_1,\ldots, \gamma_{m+1},k-l).$$
Then, we can see from (\ref{gs-g-2}) that  
\begin{equation}\label{gs-26}
\begin{array}{l}
(-1)^{m}\sum\limits_{l=1}^{s-2}A_{k-l}A_k^{-1}g_{k-s+l}^{(m)}\\
=(-1)^m\sum\limits_{l=1}^{s-2}\sum\limits_{(\gamma_1,\ldots,\gamma_{m+1})\in\mathcal{I}_{(m,k-s+l)}}(A_{k-l}A_k^{-1}A_{\gamma_1}A_k^{-1}\cdots A_{k}^{-1}A_{\gamma_m}A_{k}^{-1}A_{\gamma_{m+1}})\\
=(-1)^mg_{k-s}^{(m+1)}.\\
\end{array}
\end{equation}
Substituting (\ref{gs-25}) and (\ref{gs-26}) into (\ref{gs-24}) leads to 
$$\begin{array}{ll}
A_{k-s}&=-A_kXJ^{k+s-1}\Gamma X^{\star}A_k+g_{k-s}^{(1)}+\sum_{m=1}^{s-2}(-1)^mg_{k-s}^{(m+1)}\\
       &=-A_kXJ^{k+s-1}\Gamma X^{\star}A_k+\sum_{m=0}^{s-2}(-1)^{m}g_{k-s}^{(m+1)},\\
       &=-A_kXJ^{k+s-1}\Gamma X^{\star}A_k+\sum_{m=1}^{s-1}(-1)^{m-1}g_{k-s}^{(m)},\\
       &=-A_kXJ^{k+s-1}\Gamma X^{\star}A_k+g_{k-s},\\
\end{array}$$
i.e., $A_{k-s}$ can also be rewritten as in (\ref{gs-16}). This completes the proof. 
\end{proof}

\begin{example}
Consider the regular $(\star, \epsilon_1, \epsilon_2)$-structured matrix polynomial:
$$\lambda^5A_5+\lambda^4A_4+\lambda^3A_3+\lambda^2A_2+\lambda A_1+A_0,$$
which has a standard pair $(X,\ J)$. 
For lack of space, we only give the expression of $A_1$. By Theorem \ref{thm-1}, we have
$$A_1 =-A_5XJ^8\Gamma X^{\star}A_5-A_4XJ^7\Gamma X^{\star}A_5-A_3XJ^6\Gamma X^{\star}A_5-A_2XJ^5\Gamma X^{\star}A_5.$$ 
According to the proof of Corollary \ref{cor-1}, we know that the matrix $A_1$ can be expressed as
$$\begin{array}{ll}
A_1=& -A_5XJ^8\Gamma X^{\star}A_5+A_4A_5^{-1}A_4A_5^{-1}A_4A_5^{-1}A_4-A_4A_5^{-1}A_4A_5^{-1}A_3\\
&-A_4A_5^{-1}A_3A_5^{-1}A_4-A_3A_5^{-1}A_4A_5^{-1}A_3+A_4A_5^{-1}A_2+A_2A_5^{-1}A_4.\\
\end{array}$$
By (\ref{gs-15}), it is easy to verify that 
$$\mathcal{I}_{(3, 1)}=\{(4,4,4,4)\},\ \mathcal{I}_{(2,1)}=\{(4,4,3),(4,3,4),(3,4,4)\},\ \mathcal{I}_{(1,1)}=\{(2,4),(4,2)\}.$$
Then, we have 
$$\left\{\begin{array}{l}
g_1^{(3)}=A_4A_5^{-1}A_4A_5^{-1}A_4A_5^{-1}A_4,\\
g_1^{(2)}=A_4A_5^{-1}A_4A_5^{-1}A_3+A_4A_5^{-1}A_3A_5^{-1}A_4+A_3A_5^{-1}A_4A_5^{-1}A_3,\\
g_1^{(1)}=A_4A_5^{-1}A_2+A_2A_5^{-1}A_4.\\
\end{array}\right.$$
It follows that 
$$A_1=-A_5XJ^8\Gamma X^{\star}A_5+g_1^{(3)}-g_1^{(2)}+g_1^{(1)}=-A_5XJ^8\Gamma X^{\star}A_5+g_1,$$
where $g_1$ is defined by (\ref{gs-g-1}). 
\end{example}

\begin{remark}
(Recovery of results in Chu and Xu \cite{Chu-MC-2009} and Jia and Wei \cite{Jia-siam-2011}) Given a real quadratic $\lambda$-matrix $Q(\lambda):=\lambda^2A_2+\lambda A_1+A_0$, where $A_2, A_1, A_0$ are symmetric matrices in $\mathbb{R}^{n\times n}$ and $A_2$ is nonsingular. A sufficient and necessary provided in \cite{Chu-MC-2009} that a real matrix pair $(\mathfrak{X}, \mathfrak{J})\in\mathbb{R}^{n\times 2n}\times\mathbb{R}^{2n\times 2n}$ is a real standard pair of $Q(\lambda)$, is: there exists a symmetric and nonsingular matrix $S\in\mathbb{R}^{2n\times 2n}$ such that
\begin{align}
&\mathfrak{X}S\mathfrak{X}^T=0,\label{xzgs-2}\\
&S^{-1}\mathfrak{J}=\mathfrak{J}^TS^{-1}.\label{xzgs-3}
\end{align}
And the matrix coefficients $(A_2,A_1,A_0)$ can be factorized in terms of $(\mathfrak{X},\mathfrak{J})$ as follows:
\begin{equation}\label{xzgs-1}
\left\{\begin{array}{l}
A_2=(\mathfrak{X}\mathfrak{J}S\mathfrak{X}^T)^{-1},\\
A_1=-A_2\mathfrak{X}\mathfrak{J}^2S\mathfrak{X}^TA_2,\\
A_0=-A_2\mathfrak{X}\mathfrak{J}^3S\mathfrak{X}^TA_2+A_1A_2^{-1}A_1.\\
\end{array}\right.
\end{equation}

Substituting $\star=T$, $\epsilon_1=1$ and $\epsilon_2=1$ into (\ref{gs-5}), it is easy to verify that $S\in\mathcal{S}_{(\mathfrak{J}, T, 1, 1)}$. Clearly, the condition (\ref{xzgs-2}) provided in \cite{Chu-MC-2009} coincide with (\ref{gs-6}) by setting $k=2$. Moreover, we can see from (\ref{gs-15}) that $\mathcal{I}_{(1,0)}=\{(1,1)\}$. Then the expression of $A_0$ in (\ref{xzgs-1}) can be rewritten as $A_0=-A_2\mathfrak{X}\mathfrak{J}^3S\mathfrak{X}^TA_2+g_0$, where $g_0=A_1A_2^{-1}A_1$. Therefore, the expressions in (\ref{xzgs-1}) coincide with the expressions as those given in Corollary \ref{cor-1} by setting $k=2$ and $\star=T$.

Let $G(\lambda):=\lambda^2M+\lambda D+K$, where $M, D, K\in\mathbb{R}^{n\times n}$, $M^T=M$, $D^T=-D$, $K^T=K$, and $M$ is nonsingular. Clearly, $G(\lambda)$ is an $T$-even. Similarly, it is easy to verify that the real spectral decomposition of $G(\lambda)$ given in \cite{Jia-siam-2011} coincides with the results obtained in this paper as given in Theorem \ref{thm-1} and Corollary (\ref{cor-1}) by setting $\mathbb{K}=\mathbb{R}$, $\star=T$, $k=2$, $\epsilon_1=1$ and $\epsilon_2=-1$. 
\end{remark}

\subsection{Singular matrix polynomial}\label{sb-1}

Note that $P(\lambda)$ is singular if the coefficient matrix $A_k$ is singular. And in this case, $P(\lambda)$ has infinity eigenvalue. For simplicity, we assume that the coefficient matrix $A_0$ of $P(\lambda)$ is nonsingular, i.e., all the finite eigenvalues of $P(\lambda)$ are nonzero. Let
\begin{align}
&X=[X_F, X_{\infty}]\in\mathbb{K}^{n\times kn},\label{gss-1}\\
&J=\mbox{diag}(J_F^{-1}, J_{\infty})\in\mathbb{K}^{kn\times kn}, \label{gss-2}
\end{align}
where $(J_F, X_F)\in\mathbb{K}^{l\times l}\times \mathbb{K}^{n\times l}$ is the finite standard pair of $P(\lambda)$ and $(J_{\infty}, X_{\infty})\in\mathbb{K}^{(kn-l)\times (kn-l)}\times \mathbb{K}^{n\times (kn-l)}$ is the infinite standard pair of $P(\lambda)$ \cite{Lancaster-book}, $J_{\infty}$ is a Jordan matrix corresponding to the eigenvalue zero. Let
\begin{equation}\label{gss-3}
Q(\mu)=\mu^kA_0+\mu^{k-1}A_1+\cdots+\mu A_{k-1}+\mu A_k.
\end{equation}
Clearly, the nonzero eigenvalue $\mu$ of $Q(\mu)$ equals to $\lambda^{-1}$, where $\lambda$ is the nonzero finite  eigenvalue of $P(\lambda)$. It follows from the Theorem 7.2 of \cite{Lancaster-book} that $(J_{\infty}, X_{\infty})$ is an infinite standard pair of $P(\lambda)$ if and only if $(J_{\infty}, X_{\infty})$ is a finite standard pair of $Q(\mu)$. Note that $Q(\mu)$ is regular, since $A_0$ is nonsingular. Then, the spectral decomposition of singular matrix polynomial $P(\lambda)$ is equivalent to the spectral decomposition of regular matrix polynomial $Q(\mu)$. It follows from Definition \ref{lem-1} that $(J, X)$ is a standard pair of $Q(\mu)$ if and only if $X_L$ is nonsingular and 
\begin{equation}\label{gss-7}
A_0XJ^k+A_1XJ^{k-1}+\cdots+A_{k-1}XJ+A_kX=0.
\end{equation}

 For any matrix $\Phi\in\mathbb{K}^{m\times m}$, let
\begin{equation}\label{gs-5-0}
\mathcal{D}_{(\Phi,\star,\epsilon_1,\epsilon_2)}=\left\{S\in\mathbb{K}^{m\times m}|S^{\star}=\epsilon_1S,\ S\Phi^{\star}=\epsilon\Phi S\right\}.
\end{equation}  
Define
$$\mathcal{T}:=(X_L^{\star}\mathcal{H}X_L)^{-1}\in\mathbb{K}^{kn\times kn},$$
where $\mathcal{H}\in\mathbb{K}^{kn\times kn}$ satisfies: if $k$ is even, then $\mathcal{H}^{\star}=\epsilon_2\mathcal{H}$ and 
$${
\mathcal{H}=\begin{bmatrix}
A_{k-1} & A_{k-2} & A_{k-3}  & \cdots & A_{1} & A_0\\
\epsilon A_{k-2} & \epsilon A_{k-3} & \epsilon A_{k-4}  & \cdots & \epsilon A_0 & 0\\
A_{k-3} & A_{k-4} & A_{k-5}  & \cdots & 0 & 0\\
 \vdots & \vdots & \vdots & \ddots & \vdots & \vdots\\
A_{1} & A_0 & 0  & \cdots & 0 & 0\\
\epsilon A_0 & 0 & 0  & \cdots & 0 & 0\\
\end{bmatrix},}
$$
and if $k$ is odd, then $\mathcal{H}^{\star}=\epsilon_1\mathcal{H}$ and 
$${
\mathcal{H}=\begin{bmatrix}
A_{k-1} & A_{k-2} & A_{k-3}  & \cdots & A_{1} & A_0\\
\epsilon A_{k-2} & \epsilon A_{k-3} & \epsilon A_{k-4}  & \cdots & \epsilon A_0 & 0\\
A_{k-3} & A_{k-4} & A_{k-5}  & \cdots & 0 & 0\\
 \vdots & \vdots & \vdots & \ddots & \vdots & \vdots\\
\epsilon A_{1} & \epsilon A_0 & 0  & \cdots & 0 & 0\\
 A_0 & 0 & 0  & \cdots & 0 & 0\\
\end{bmatrix}.}
$$ 
Similar to the proof of Theorem \ref{thm-1}, we can obtain the following results. 
\begin{theorem}\label{thm-2}
Suppose that $k$ is even. Then $(X,\ J)$ defined in (\ref{gss-1}) and (\ref{gss-2}) is a standard pair of regular  $Q(\mu)$  in (\ref{gss-3}), i.e., (\ref{gss-7}) holds if and only if the matrix $X_L$ defined in (\ref{xgs-1}) and $XJ^{k-1}\mathcal{T} X^{\star}$ are nonsingular, and there exists an $kn\times kn$ nonsingular matrix $\mathcal{T}$ satisfying $\mathcal{T}\in\mathcal{S}_{(J,\star,\epsilon_1,\epsilon_2)}$ and 
\begin{equation*}
XJ^r\mathcal{T} X^{\star}=0, \ r=0, 1, \ldots, k-2.
\end{equation*}
And in this case, the coefficient matrices $A_0, A_1, \ldots, A_k$ of $Q(\mu)$ can be expressed as 
\begin{equation*}
\left\{\begin{array}{l}
A_0=(XJ^{k-1}\mathcal{T} X^{\star})^{-1},\\
A_{k-j}=-\sum_{i=j+1}^kA_iXJ^{k-j+i-1}\mathcal{T} X^{\star}A_0,\ j=0, 1, \ldots, k-1.
\end{array}\right.
\end{equation*}
\end{theorem}

\begin{theorem}\label{thm-3}
Suppose that $k$ is odd. Then $(X,\ J)$ defined in (\ref{gss-1}) and (\ref{gss-2}) is a standard pair of regular  $Q(\mu)$ in (\ref{gss-3}), i.e., (\ref{gss-7}) holds if and only if the matrix $X_L$ defined in (\ref{xgs-1}) and $XJ^{k-1}\mathcal{T} X^{\star}$ are nonsingular, and there exists an $kn\times kn$ nonsingular matrix $\mathcal{T}$ satisfying $\mathcal{T}\in\mathcal{D}_{(J,\star,\epsilon_1,\epsilon_2)}$ and 
\begin{equation*}
XJ^r\mathcal{T} X^{\star}=0, \ r=0, 1, \ldots, k-2.
\end{equation*}
And in this case, the coefficient matrices $A_0, A_1, \ldots, A_k$ of $Q(\mu)$ can be expressed as 
\begin{equation*}
\left\{\begin{array}{l}
A_0=(XJ^{k-1}\mathcal{T} X^{\star})^{-1},\\
A_{k-j}=-\sum_{i=j+1}^kA_iXJ^{k-j+i-1}\mathcal{T} X^{\star}A_0,\ j=0, 1, \ldots, k-1.
\end{array}\right.
\end{equation*}
\end{theorem}
\begin{corollary}\label{cor-2}
The matrices $A_0, A_1, \ldots, A_{k}$ given by Theorem \ref{thm-2} and Theorem \ref{thm-3} can be equivalently rewritten as
\begin{equation*}
\left\{\begin{array}{l}
A_0=(XJ^{k-1}\mathcal{T} X^{\star})^{-1},\\
A_{1}=-A_0XJ^k\mathcal{T}  X^{\star}A_0,\\
A_{k-j}=-A_0XJ^{2k-j-1}\mathcal{T} X^{\star}A_0+h_j, \ \ j=0, 1, \ldots, k-2,\\
\end{array}\right.
\end{equation*}
where 
\begin{align*}
&h_j=\sum_{t=1}^{k-j-1}(-1)^{t-1}h_j^{(t)},\\
&h_j^{(t)}=\sum_{(\gamma_1,\ldots,\gamma_{t+1})\in\mathcal{I}_{(t,j)}}(A_{\gamma_1}A_0^{-1}A_{\gamma_2}A_{0}^{-1}\cdots A_{0}^{-1}A_{\gamma_t}A_{0}^{-1}A_{\gamma_{t+1}}).
\end{align*}
\end{corollary}

\section{Structures of $\mathcal{S}_{(J,\star,\epsilon_1,\epsilon_2)}$ for regular $P(\lambda)$}
Throughout of this section, we always assume that the coefficient matrix $A_k$ of $P(\lambda)$ is nonsingular. In this section, we discuss the structures of $\mathcal{S}_{(J,\star,\epsilon_1,\epsilon_2)}$. 
Let $J(\lambda_j)=\lambda_jI_{n_j}+N_j$ be the Jordan canonical form associated with the eigenvalue $\lambda_j$, which may contain several Jordan blocks. The matrix $N_j$ is an $n_j\times n_j$ nilpotent matrix with ones or zeros along its superdiagonal, depending on the partial multiplicities of $\lambda_j$. According to the Table \ref{tab1}, we separate the distinct eigenvalues $\lambda_j$ and corresponding Jordan canonical form into the following four categories.

Case 1. $T$-symmetric/$T$-skew-symmetric ($\mathbb{K}=\mathbb{R}$) 
\begin{align}
&J_j=\begin{bmatrix}
\alpha_jI_{n_j}+N_j & \beta_jI_{n_j}\\
-\beta_jI_{n_j}& \alpha_jI_{n_j}+N_j\\
\end{bmatrix},\ \lambda_j\in\mathbb{C}/\mathbb{R},\ j=1,\ldots, r,\label{xxg-3}\\
&J_j=\alpha_jI_{n_j}+N_j,\ \alpha_j\in\mathbb{R},\ j=r+1,\ldots,s.\label{xxg-4}
\end{align}

Case 2. $T$-even/$T$-odd ($\mathbb{K}=\mathbb{R}$)
\begin{align}
&{J_j=\mbox{diag}\left(\begin{bmatrix}
\alpha_jI_{n_j}+N_j & \beta_jI_{n_j}\\
 -\beta_jI_{n_j} & \alpha_jI_{n_j}+N_j\\
 \end{bmatrix}, \begin{bmatrix}
-\alpha_jI_{n_j}+N_j & -\beta_jI_{n_j}\\
 \beta_jI_{n_j} & -\alpha_jI_{n_j}+N_j\\
 \end{bmatrix}\right),j=1,\ldots,r_1;}\label{szgs-6-1}\\
&J_j=\begin{bmatrix}
N_j & \beta_jI_{n_j}\\
-\beta_jI_{n_j} & N_j\\
\end{bmatrix}, j=r_1+1,\ldots, r_2,\label{szgs-6-2}\\
&J_j=\begin{bmatrix}
\alpha_jI_{n_j}+N_j & 0\\
0 & -\alpha_jI_{n_j}+N_j\\
\end{bmatrix}, j=r_2+1, \ldots,s-1;\label{szgs-6-3}\\
& J_{s}=N_{s}.\label{szgs-6-4}
\end{align}

Case 3. $H$-Hermitian/$H$-skew-Hermitian ($\mathbb{K}=\mathbb{C}$)
\begin{align}
&J_j=\mbox{diag}(\lambda_jI_{n_j}+N_j, \bar{\lambda}_jI_{n_j}+N_j), \lambda_j\in\mathbb{C}/\mathbb{R},\ j=1,\ldots,r;\label{szgs-7-1}\\
&J_{j}=\lambda_jI_{n_j}+N_j, \lambda_j\in\mathbb{R},\ j=r+1,\ldots,s.\label{szgs-7-2}
\end{align}

Case 4. $H$-even/$H$-odd ($\mathbb{K}=\mathbb{C}$)
\begin{align}
&J_j=\mbox{diag}(\lambda_jI_{n_j}+N_j, -\bar{\lambda}_jI_{n_j}+N_j), \lambda_j\in\mathbb{C}/\mathbb{{\rm i}R},\ j=1,\ldots,r;\label{szgs-8-1}\\
&J_{j}=\lambda_jI_{n_j}+N_j, \lambda_j\in\mathbb{{\rm i}R},\ j=r+1,\ldots,s,\label{szgs-8-2}
\end{align}
where ${\rm i}:=\sqrt{-1}$.

The structures of $\Gamma\in\mathcal{S}_{(J,\star,\epsilon_1,\epsilon_2)}$ for the $T$-symmetric  and $T$-even  have been given by \cite{Chu-MC-2009,Jia-siam-2011}. In this section, we consider the other six $(\star,\epsilon_1,\epsilon_2)$-structured matrix polynomials.  Let $(X,\ J)\in\mathbb{K}^{n\times kn}\times\mathbb{K}^{kn\times kn}$ be a standard pair of $P(\lambda)$, where 
\begin{align}
&X:=[X_1,X_2,\ldots,X_s],\label{gss-8-1}\\
&J:=\mbox{diag}(J_1,J_2,\ldots,J_s).\label{gss-8-2}
\end{align}

\begin{lemma}\label{lem-3}
Suppose that matrix pair $(X,\ J)$ defined by (\ref{gss-8-1}) and (\ref{gss-8-2}) is a standard pair of regular $P(\lambda)$ as in (\ref{gs-1}). Then $P(\lambda)$ has a spectral decomposition (\ref{gs-7}), in which the matrix $\Gamma\in\mathcal{S}_{(J,\star,\epsilon_1,\epsilon_2)}$ has the following structure
\begin{equation}\label{gss-10}
\Gamma=\mbox{diag}(\Gamma_{11}, \Gamma_{22}, \ldots, \Gamma_{ss}),
\end{equation}
where $\Gamma_{jj}\in\mathcal{S}_{(J_j,\star,\epsilon_1,\epsilon_2)}$, $j=1, \ldots, s$.
\end{lemma}
\begin{proof}
Since $(X,\ J)$ is the standard pair of $P(\lambda)$, it follows from Theorem \ref{thm-1} that there exists a nonsingular matrix $\Gamma\in\mathcal{S}_{(J,\star,\epsilon_1,\epsilon_2)}$ such that (\ref{gs-6}) and (\ref{gs-7}) hold. Then,
\begin{align}
&\Gamma^{\star}=\epsilon_2\Gamma,\label{gss-11-1}\\
&\Gamma J^{\star}=\epsilon J\Gamma.\label{gss-11-2}
\end{align}
Partition $\Gamma$ as $\Gamma:=(\Gamma_{jl})$ according to the matrix $J$ defined by (\ref{gss-8-2}). It follows from (\ref{gss-11-2}) that 
\begin{equation}\label{gss-13}
\Gamma_{jl}J^{\star}_l-\epsilon J_j\Gamma_{jl}=0, \ \ j, l=1, 2, \ldots, s,
\end{equation}
 which implies that
\begin{equation}\label{gss-12}
\left[(J^{\star}_l)^T\otimes I-\epsilon I\otimes J_j\right]\mbox{vec}(\Gamma_{jl})=0,
\end{equation}
where $\otimes $ is the Kronecker product and $\mbox{vec}$ the column vectorization of a matrix. When $j\neq l$, it is easy to verify that the matrix $(J^{\star}_l)^T\otimes I-\epsilon I\otimes J_j$ is nonsingular, since the eigenvalues of $J_l^{\star}$ and $\epsilon J_j$ are distinct. And then $\mbox{vec}(\Gamma_{jl})=0$, so $\Gamma_{jl}=0$ ($j\neq l$). From (\ref{gss-11-1}) and (\ref{gss-11-2}), we can obtain from (\ref{gs-5}) that $\Gamma_{jj}\in\mathcal{S}_{(J_j,\star,\epsilon)}$, $j=1, \ldots, s$.
\end{proof}


\subsection{Upper triangular Hankel/skew-Hankel blocks form}
An $m\times n$ matrix $H=(h_{ij})$ is said to have a Hankel structure if $h_{ij}=\eta_{i+j-1}$, where $\{\eta_1,\ldots,\eta_{m+n-1}\}$ are some fixed scalars. The matrix $H$ is said to be upper triangular Hankel if $\eta_t=0$ for all $t>\min\{m,n\}$. For the matrix $$D_m=\mbox{diag}(1,-1, \ldots, (-1)^{m-1}),$$ a matrix $H$ is called a skew-Hankel matrix if $D_mH$ is a Hankel matrix; $H$ is an upper triangular skew-Hankel matrix if $D_mH$ is an upper triangular Hankel matrix. Assume that there are $m_j$ Jordan blocks corresponding to the eigenvalues $\lambda_j$. Then the nilpotent matrix $N_j$ has the following form
$$N_j=\mbox{diag}(N_1^{(j)}, N_2^{(j)}, \ldots, N_{m_j}^{(j)}),$$
where $N_t^{(j)}$ is the nilpotent block of size $n_t^{(j)}$, $t=1,\ldots, m_j$. By straightforward calculation, we can obtain the following results, which can also be found in \cite{Chu-MC-2009,Jia-siam-2011}.
\begin{lemma}\cite{Chu-MC-2009,Jia-siam-2011}\label{lem-4}
{\rm (1)} Any solution $Z$ to the equation 
\begin{equation}\label{gss-14-1}
ZN_j^T-N_jZ=0,
\end{equation}
is necessarily upper triangular Hankel blocks of form
\begin{equation}\label{gss-14-3}
Z=[Z_{it}]_{m_j\times m_j},\ \ Z_{it}\in\mathbb{R}^{n_i^{(j)}\times n_t^{(j)}},
\end{equation}
where $Z_{it}$ are upper triangular Hankel for $i,t=1,\ldots, m_j$. If $Z$ is $T$-symmetric, then $Z_{it}=Z_{it}^T$, where $Z_{ii}$ is a $T$-symmetric upper triangular Hankel matrix. If $Z$ is $T$-skew-symmetric, then  $Z_{it}=-Z_{it}^T$, where $Z_{ii}$ is a $T$-skew-symmetric upper triangular Hankel matrix.

{\rm (2)} Any  solution $Z$ to the euqation
\begin{equation}\label{gss-14-2}
ZN_j^T+N_jZ=0,
\end{equation}
is necessarily upper triangular skew-Hankel blocks form as in (\ref{gss-14-3}), where $Z_{it}$ are upper triangular skew-Hankel for $i,t=1,\ldots,m_j$. If $Z$ is $T$-symmetric, then $Z_{it}=Z_{it}^T$, where $Z_{ii}$ is a $T$-symmetric upper triangular skew-Hankel matrix (whose $(i,t)$-element is zero if $i+t$ is odd). If $Z$ is $T$-skew-symmetric, then  $Z_{it}=-Z_{it}^T$, where $Z_{ii}$ is a $T$-skew-symmetric upper triangular skew-Hankel matrix (whose $(i,t)$-element is zero if $i+t$ is even).
\end{lemma}

\begin{theorem}\label{thm-4}
Suppose that the matrix pair $(X,\ J)$ defined by (\ref{gss-8-1}) and (\ref{gss-8-2}) is a standard pair of regular $P(\lambda)$ as in (\ref{gs-1}). Then $P(\lambda)$ has a spectral decomposition (\ref{gs-7}), in which the matrix $$\Gamma=\mbox{diag}(\Gamma_{11}, \Gamma_{22}, \ldots, \Gamma_{ss})\in\mathcal{S}_{(J,\star,\epsilon_1,\epsilon_2)},$$ has the following structures:

(1) $T$-skew-symmetric:
\begin{align}
&\Gamma_{jj}=\begin{bmatrix}
U_{jj} & W_{jj}\\
W_{jj} & -U_{jj}\\
\end{bmatrix},\ \  j=1,2,\ldots,r,\label{xxg-2}\\
&\Gamma_{jj}=U_{jj},\ \  j=r+1,\ldots,s,\label{xxg-2-1}
\end{align}
where $U_{jj}, W_{jj}\in\mathbb{R}^{n_j\times n_j}$ are of $T$-skew-symmetric upper triangular Hankel block form.

(2) $T$-odd:
\begin{align}
& \Gamma_{jj}=\begin{bmatrix}
0 & \tilde{U}_j\\
\tilde{U}_j^T &0\\
\end{bmatrix}, \tilde{U}_j=\begin{bmatrix}
\tilde{U}_{j1} & \tilde{U}_{j2}\\
\tilde{U}_{j2} & -\tilde{U}_{j1}\\
\end{bmatrix}\in\mathbb{R}^{2n_j\times 2n_j}, j=1,\ldots, r_1\label{xzgs-5-1}\\
& \Gamma_{jj}=\begin{bmatrix}
\tilde{W}_{j1} & \tilde{W}_{j2}\\
-\tilde{W}_{j2} & \tilde{W}_{j1}\\
\end{bmatrix}\in\mathbb{R}^{2n_j\times 2n_j}, j=r_1+1,\ldots,r_2,\label{xzgs-5-2}\\
& \Gamma_{jj}=\begin{bmatrix}
0 & \tilde{V}_{j2}\\
\tilde{V}_{j2}^T & 0\\
\end{bmatrix}\in\mathbb{R}^{2n_j\times 2n_j}, j=r_2+1, \ldots, s-1,\label{xzgs-5-3}
\end{align}
where $\tilde{U}_{j1}$, $\tilde{U}_{j2}$, $\tilde{V}_{j2}$ are of upper triangular skew-Hankel blocks form,  and $\tilde{W}_{j1}$, $\Gamma_{ss}$ are of $T$-symmetric upper triangular skew-Hankel blocks form and $\tilde{W}_{j2}$ is of $T$-skew-symmetric upper triangular skew-Hankel blocks form. 

(3) $\star=H$ and $\epsilon_1=\epsilon_2$:
\begin{align*}
&\Gamma_{jj}=\begin{bmatrix}
0 & W_{j}\\
\rho W_{j}^H & 0\\
\end{bmatrix},\ \ j=1,2, \ldots, r, \\
&\Gamma_{jj}= W_{j}, \ \ j=r+1,\ldots,s,
\end{align*}
where $W_{j}\in\mathbb{C}^{n_j\times n_j}$$(j=1,\ldots,r)$ are of upper triangular Hankel block form. If $P(\lambda)$ is $H$-Hermitian, then $\rho=1$ and $W_{j}\in\mathbb{C}^{n_j\times n_j}$($j=r+1,\ldots, s$) are of  $H$-Hermitian upper triangular Hankel blocks form. If $P(\lambda)$ is $H$-skew-Hermitian, then $\rho=-1$ and the matrices $W_{j}\in\mathbb{C}^{n_j\times n_j}$($j=r+1,\ldots, s$) are of $H$-skew-Hermitian upper triangular Hankel blocks form.

(4) $\star=H$ and $\epsilon_1\neq\epsilon_2$:
\begin{align*}
&\Gamma_{jj}=\begin{bmatrix}
0 & V_{j}\\
\rho V_{j}^H & 0\\
\end{bmatrix},\ \ j=1,\ldots,r,\\
&\Gamma_{jj}=V_{j},\ \ j=r+1,\ldots,s,
\end{align*}
where $V_{j}\in\mathbb{C}^{n_j\times n_j}$($j=1,\ldots,r$) are of upper triangular skew-Hankel blocks form. If $P(\lambda)$ is $H$-odd, then $\rho=1$ and  $V_{j}\in\mathbb{C}^{n_j\times n_j}$($j=r+1,\ldots,s$) are of $H$-Hermitian upper triangular skew-Hankel blocks form. If $P(\lambda)$ is $H$-even, then $\rho=-1$ and $V_{j}\in\mathbb{C}^{n_j\times n_j}$($j=r+1,\ldots,s$) are of  $H$-skew-Hermitian upper triangular skew-Hankel blocks form.
\end{theorem}
\begin{proof}
The proof of $T$-skew-symmetric and $T$-odd are similar to the $T$-symmetric and $T$-even quadratic matrix polynomials given by Chu \cite{Chu-MC-2009} and Jia \cite{Jia-siam-2011}, respectively. We only give the proof the case of $H$-Hermitian, i.e., $\star=H$, $\epsilon_1=\epsilon_2=1$. The other cases can be proved similarly. Let $\lambda_j=\alpha_j+{\rm i}\beta_j$ and $\Gamma_{jj}=\Gamma_R+{\rm i}\Gamma_I$, where $\alpha_j,\beta_j\in\mathbb{R}$. Note that $\Gamma_{jj}^{H}=\Gamma_{jj}$, we must have $\Gamma_R^T=\Gamma_R$ and $\Gamma_I^T=-\Gamma_I$. If $\beta_j=0$, then we can see from $\Gamma_{jj}J_j^{H}-J_j\Gamma_{jj}=0$ that 
\begin{align*}
&\Gamma_RN_j^T-N_j\Gamma_R=0,\\
&\Gamma_IN_j^T-N_j\Gamma_I=0.
\end{align*}
By Lemma \ref{lem-4}, we can obtain that the matrices $\Gamma_R, \Gamma_I$ are of $T$-symmetric and $T$-skew-symmetric upper triangular Hankel blocks forms, respectively.
If $\beta_j\neq 0$, we can assume that 
\begin{equation}\label{gss-15}
J_j=\begin{bmatrix}
\alpha_jI_{n_j}+N_j+{\rm i}\beta_jI_{n_j} & 0\\
0&  \alpha_jI_{n_j}+N_j-{\rm i}\beta_jI_{n_j}\\
\end{bmatrix}\in\mathbb{C}^{2n_j\times 2n_j}.
\end{equation}
For simplicity, denote the blocks of $\Gamma_{R}$ and $\Gamma_I$ of form 
\begin{equation}\label{gss-16-0}
\Gamma_{R}=\begin{bmatrix}
U_R & W_R\\
W_R^T & V_R\\
\end{bmatrix},\ \ \Gamma_I=\begin{bmatrix}
U_I & W_I\\
-W_I^T & V_I\\
\end{bmatrix},
\end{equation}
where $U_R, V_R\in\mathbb{R}^{n_j\times n_j}$ are $T$-symmetric and $U_I, V_I\in\mathbb{R}^{n_j\times n_j}$ are $T$-skew-symmetric. Comparing the corresponding blocks in $\Gamma_{jj}J_j^{H}-J_j\Gamma_{jj}=0$, we can obtain that
\begin{align}
&U_RN_j^T-N_jU_R+2\beta_jU_I=0,\label{gss-16-1}\\
&U_IN_j^T-N_jU_I-2\beta_jU_R=0,\label{gss-16-2}\\
&V_RN_j^T-NV_R-2\beta_jV_I=0,\label{gss-16-5}\\
&V_IN_j^T-NV_I+2\beta_jV_R=0,\label{gss-16-6}\\
&W_RN_j^T-N_jW_R=0,\label{gss-16-3}\\
&W_IN_j^T-N_jW_I=0,\label{gss-16-4}
\end{align}
Thus, we can obtain from (\ref{gss-16-2}) that 
\begin{equation}\label{gss-17}
2\beta_jU_R=U_IN_j^T-N_jU_I.
\end{equation}
Substituting (\ref{gss-17}) into (\ref{gss-16-1}) leads to 
$$U_I(N_j^T)^2-2N_jU_IN_j^T+N_j^2U_I+4\beta^2_jU_I=0,$$
which implies that
\begin{equation}\label{gss-18}
\left(N_j^2\otimes I_{n_j}-2N_j\otimes N_j+I_{n_j}\otimes N_j^2+4\beta_j^2I_{n_j}\otimes I_{n_j}\right)\mbox{vec}(U_I)=0.
\end{equation}
Since $\beta_j\neq 0$, it is easy to see that the matrix 
$$N_j^2\otimes I_{n_j}-2N_j\otimes N_j+I_{n_j}\otimes N_j^2+4\beta_j^2I_{n_j}\otimes I_{n_j}$$
is nonsingular. It follows from (\ref{gss-18}) that $\mbox{vec}(U_I)=0$, i.e, $U_I=0$. Similarly, we can prove that $U_R=0$, $V_R=0$ and $V_I=0$. By Lemma \ref{lem-4}, we can see from (\ref{gss-16-3}) and (\ref{gss-16-4}) that the matrices $W_R, W_I$ are of upper triangular Hankel blocks form. It follows from (\ref{gss-16-0}) that $\Gamma_{jj}=\Gamma_R+{\rm i}\Gamma_I$ is of $H$-Hermitian upper triangular Hankel blocks from. This completes the proof.
\end{proof}

\subsection{The semi-simple structure}

It is well known that an eigenvalue $\lambda_j$ is semi-simple if its algebraic multiplicity $n_j$ equal to its geometric multiplicity $m_j$. When all eigenvalues of $P(\lambda)$ are semi-simple, the upper triangular Hankel structures in Theorem \ref{thm-4} no longer show up. 

In order to simplify the structure of $\mathcal{S}_{(J,\star,\epsilon_1, \epsilon_2)}$ for the $T$-skew-symmetric and $T$-odd, we recall two definitions \cite{ARV-1992-siam}. Let $\Lambda:=\bigl[\begin{smallmatrix}
0 & I_n\\
-I_n & 0\\
\end{smallmatrix}\bigr]$. A matrix $Q\in\mathbb{R}^{2n\times 2n}$ is called orthogonal symplectic if and only if $Q\Lambda Q^T=\Lambda$ and $Q^TQ=I_{2n}$. Any orthogonal-symplectic matrix $Q$ can be written as $Q=\bigl[\begin{smallmatrix}
Q_1 & Q_2\\
-Q_2 & Q_1\\
\end{smallmatrix}\bigr]$, where $Q_1,Q_2\in\mathbb{R}^{n\times n}$. A matrix $H\in\mathbb{R}^{2n\times 2n}$ is said to be skew-Hamiltonian if $H\Lambda=-(H\Lambda)^T$.

\begin{lemma}\label{lem-5}\cite{ARV-1992-siam}
Let $H\in\mathbb{R}^{2m\times 2m}$ be a skew-Hamiltonian matrix. Then there is an orthogonal-symplectic matrix $\tilde{Q}_1\in\mathbb{R}^{2m\times 2m}$ such that  
$$\tilde{Q}_1^TH\tilde{Q}_1=\begin{bmatrix}
H_1 & H_2\\
0 & H_1^T\\
\end{bmatrix},$$
where $H_1,H_2\in\mathbb{R}^{m\times m}$, $H_2^T=-H_2$ and $H_1$ is upper quasi-triangular.
\end{lemma}

\begin{lemma}\label{lem-6}
Let $\Gamma_{jj}$ ($j=1,\ldots,r$) be given by (\ref{xxg-2}). Then $n_j$ is even and there exist orthogonal-symplectic matrices $Q_j\in\mathbb{R}^{2n_j\times 2n_j}$ such that 
\begin{equation}\label{xxg-7}Q_j^T\Gamma_{jj}Q=\mbox{diag}\left(\begin{bmatrix}
0 & D_{j}\\
-D_{j} & 0\\
\end{bmatrix}, \begin{bmatrix}
0 & -D_{j}\\
D_{j} & 0\\
\end{bmatrix}\right),
\end{equation}
where $D_{j}=\mbox{diag}(\sigma^{(j)}_1,\ldots,\sigma^{(j)}_{\frac{n_j}{2}})$, $\sigma^{(j)}_k\in\mathbb{R}$ for $k=1,\ldots,\frac{n_j}{2}$.
\end{lemma}

\begin{proof}
It is easy to verify that the matrix $\Gamma_{jj}=\bigl[\begin{smallmatrix}
U_{jj} & W_{jj}\\
W_{jj} & -U_{jj}\\
\end{smallmatrix}\bigr]$ is a skew-Hamiltonian matrix, it follows from Lemma \ref{lem-5} that there is an orthogonal-symplectic matrix $\hat{Q}_j\in\mathbb{R}^{2n_j\times 2n_j}$ such that 
$$\hat{Q}_j^TH\hat{Q}_j=\begin{bmatrix}
H_{j1} & H_{j2}\\
0 & H_{j1}^T\\
\end{bmatrix},$$
where $H_{j1}, H_{j2}\in\mathbb{R}^{n_j\times n_j}$, $H_{j2}^T=-H_{j2}$ and $H_{j1}$ is upper quasi-triangular. Note that $\Gamma_{jj}$ is nonsingular and skew-symmetric, then $H_{j2}=0$ and $n_j$ is even. Moreover, we have $$H_{j1}=\mbox{diag}\left(\begin{bmatrix}
0 & \sigma^{(j)}_1\\
-\sigma^{(j)}_1 & 0\\
\end{bmatrix},\begin{bmatrix}
0 & \sigma^{(j)}_2\\
-\sigma^{(j)}_2 & 0\\
\end{bmatrix},\cdots,\begin{bmatrix}
0 & \sigma^{(j)}_{\frac{n_j}{2}}\\
-\sigma^{(j)}_{\frac{n_j}{2}} & 0\\
\end{bmatrix}\right),$$
where $\sigma^{(j)}_1, \ldots, \sigma^{(j)}_{\frac{n_j}{2}}\in\mathbb{R}$ are nonzero. Clearly, there exists a permutation matrix $\hat{Q}_{j2}\in\mathbb{R}^{n_j\times n_j}$ such that 
$$\hat{Q}_{j2}^TH_{j1}\hat{Q}_{j2}=\begin{bmatrix}
0 & D_{j}\\
-D_{j} & 0\\
\end{bmatrix},$$
where $D_{j}=\mbox{diag}(\sigma^{(j)}_1,\ldots,\sigma^{(j)}_{\frac{n_j}{2}})$. Let $Q_j=\hat{Q}_{j1}\bigl[\begin{smallmatrix}
\hat{Q}_{j2} & 0\\
0 & \hat{Q}_{j2}\\
\end{smallmatrix}\bigr]$. It is easy to verify that $Q_j$ is an orthogonal-symplectic matrix which satisfies (\ref{xxg-7}). 
\end{proof}

\begin{lemma}\label{lem-7}\cite{Jia-siam-2011}
Let $M\in\mathbb{R}^{2m\times 2m}$ be a symmetric skew-Hamiltonian matrix. Then there is an orthogonal-symplectic matrix $\tilde{Q}_2\in\mathbb{R}^{2m\times 2m}$ such that  
$$\tilde{Q}_2^TM\tilde{Q}_2=\begin{bmatrix}
D & 0\\
0 & D\\
\end{bmatrix},$$
where $D=\mbox{diag}(\delta_1,\ldots,\delta_m)\in\mathbb{R}^{m\times m}$.
\end{lemma}

\begin{lemma}\label{lem-8}
Let $U=\bigl[\begin{smallmatrix}
U_1 & U_2\\
U_2 &-U_1\\
\end{smallmatrix}\bigr]$ be a nonsingular matrix, where $U_1, U_2\in\mathbb{R}^{m\times m}$. There exist two  orthogonal-symplectic matrices $P_1$ and $P_2$ in $\mathbb{R}^{2m\times 2m}$ such that 
\begin{equation}\label{xzgs-4}
P_1^TUP_2=\begin{bmatrix}
\hat{D} & 0\\
0 & -\hat{D}\\
\end{bmatrix},\ \ \hat{D}=\mbox{diag}(\sigma_1,\ldots,\sigma_m), \ \ \sigma_k>0, \ k=1,\ldots,m.
\end{equation}
\end{lemma}
\begin{proof}
By the  Algorithm 4.4 in \cite{Benner-1998}, we can find two orthogonal-symplectic matrices $\hat{P}_1$ and $\hat{P}_2$ such that 
$$\hat{P}_1^TU\hat{P}_2=\begin{bmatrix}
M_1 & 0\\
0 & -M_1\\
\end{bmatrix},$$
where $M_1$ is upper triangular and $M_1^T$ is upper Hessenberg. Let the SVD of $M_1$ be $M_1=\hat{P}_3\hat{D}\hat{P}_4^T$, where $\hat{P}_3, \hat{P}_4\in\mathbb{R}^{m\times m}$ are orthogonal matrices and  $\hat{D}=\mbox{diag}(\sigma_1,\ldots,\sigma_m), \sigma_k>0, k=1,\ldots,m$. Let $P_1=\hat{P}_1\bigl[\begin{smallmatrix}
\hat{P}_3 & 0\\
0 &\hat{P}_3\\
\end{smallmatrix}\bigr]$ and $P_2=\hat{P}_2\bigl[\begin{smallmatrix}
\hat{P}_4 & 0\\
0 &\hat{P}_4\\
\end{smallmatrix}\bigr]$. It is easy to verify that $P_1$ and $P_2$ are orthogonal-symplectic matrices which satisfy (\ref{xzgs-4}).\end{proof}

With Lemma \ref{lem-6}, Lemma \ref{lem-7} and Lemma \ref{lem-8}, we now specify the simplest possible structure of $\Gamma\in\mathcal{S}_{(J,T,-1,\epsilon_2)}$ when all eigenvalues of $T$-skew-symmetric and $T$-odd are semi-simple. 

\begin{theorem}\label{thm-T-s}
Suppose that all eigenvalues of the $T$-skew-symmetric matrix polynomial $P(\lambda)$ are semi-simple with a real standard pair $(X, J)$ given by (\ref{gss-8-1}) and (\ref{gss-8-2}). Then, the algebraic multiplicity $n_j$ of eigenvalue $\lambda_j$ is even, i.e., $n_j=2m_j$, $m_j\in\mathbb{Z}^+$ for $j=1,\ldots, s$, and  there exists a real standard pair 
$(\tilde{X}, J)$ such that the corresponding matrix $\tilde{\Gamma}\in\mathcal{S}_{(J,T,-1,-1)}$ has the structure
\begin{equation}\label{xxg-1}
\begin{array}{ll}
\tilde{\Gamma}=&\mbox{diag}\left(\begin{bmatrix}
0 & E_{m_1}\\
-E_{m_1} & 0\\
\end{bmatrix}, \begin{bmatrix}
0 & -E_{m_1}\\
E_{m_1} & 0\\
\end{bmatrix},\cdots,\begin{bmatrix}
0 & E_{m_r}\\
-E_{m_r} & 0\\
\end{bmatrix}, \begin{bmatrix}
0 & -E_{m_r}\\
E_{m_r} & 0\\
\end{bmatrix},\right.\\
 &\ \ \ \ \ \ \ \  \ \ \left.\begin{bmatrix}
0 & I_{m_{r+1}}\\
- I_{m_{r+1}}  & 0\\
\end{bmatrix},\cdots,\begin{bmatrix}
0 & I_{m_{s}}\\
- I_{m_{s}}  & 0\\
\end{bmatrix}\right),\\
\end{array}
\end{equation} 
where  $E_{m_j}=\mbox{diag}(\pm1,\pm1,\ldots,\pm1)$ is of order $m_j\times m_j$, $j=1,\ldots,r.$
\end{theorem}
\begin{proof}
We discuss the structures of $\Gamma_{jj}$ in (\ref{xxg-2}) and (\ref{xxg-2-1}) in two different cases mentioned at the beginning of this section.

Case 1. $\lambda_j=\alpha_j+i\beta_j$ with $\lambda_j\in\mathbb{C}/\mathbb{R}$, and $J_j$ and $\Gamma_{jj}$
are defined by (\ref{xxg-3}) and (\ref{xxg-2}), respectively, for $j=1,\ldots,r$.  Observe from (\ref{xxg-2}) that for $j=1,\ldots,r$, the skew-symmetric matrix $\Gamma_{jj}$ is in fact skew-Hamiltonian. By Lemma \ref{lem-6}, we have $n_j=2m_j$, $m_j\in\mathbb{Z}^+$,  and there exist an orthogonal-symplectic matrix $Q_j=\bigl[\begin{smallmatrix}
Q_{j1} & Q_{j2}\\
-Q_{j2} & Q_{j1}\\
\end{smallmatrix}\bigr]\in\mathbb{R}^{2n_j\times 2n_j}$ such that 
$$Q_j^T\Gamma_{jj}Q=\mbox{diag}\left(\begin{bmatrix}
0 & D_{j}\\
-D_{j} & 0\\
\end{bmatrix}, \begin{bmatrix}
0 & -D_{j}\\
D_{j} & 0\\
\end{bmatrix}\right),,$$
where $D_{j}=\mbox{diag}(\sigma^{(j)}_1,\ldots,\sigma^{(j)}_{m_j})$, $\sigma^{(j)}_k\in\mathbb{R}$ for $k=1,\ldots,m_j$. Let $D_{jj}=\mbox{diag}(|D_j|,|D_j|,|D_j|,|D_j|)^{1/2}$. Note that $Q_{j1}+{\rm i}Q_{j2}$ is an unitary matrix in $\mathbb{C}^{n_j\times n_j}$. Columns of the complex-valued $n\times n_j$ matrix 
$$\tilde{X}_j:=X_j(Q_{j1}+{\rm i}Q_{j2})D_{jj}$$
remain to represent eigenvectors of $P(\lambda)$ with corresponding to eigenvalue $\lambda_j$. Based on (\ref{gss-8-1}), we can identify the real and the imaginary parts of $\tilde{X}_j$ as 
$$\begin{array}{ll}
\tilde{X}_j&=\tilde{X}_{jR}+{\rm i}\tilde{X}_{jI}\\
& =[(X_{jR}Q_{j1}-X_{jI}Q_{j2})\mbox{diag}(|D_{j}|^{1/2},|D_j|^{1/2}),\ (X_{jR}Q_{j2}+X_{jI}Q_{j1})\mbox{diag}(|D_{j}|^{1/2},|D_j|^{1/2})]\\
&=[X_{jR},\ X_{jI}]Q_jD_{jj}.\\
\end{array}$$
It follows that $(\tilde{X}_j, J_j)$ is a real pair of $P(\lambda)$. 
Since $N_j=0$ by assumption, it is easy to verify that $Q_j^TJ_jQ_j=J_j$ for $j=1,\ldots, r$, and 
\begin{equation}\label{xxg-8}
\tilde{\Gamma}_{jj}=D_{jj}^{-1}Q_j^T\Gamma_{jj}Q_jD_{jj}^{-1}=\mbox{diag}\left(\begin{bmatrix}
0 & E_{m_j}\\
-E_{m_j} & 0\\
\end{bmatrix}, \begin{bmatrix}
0 & -E_{m_j}\\
E_{m_j} & 0\\
\end{bmatrix}\right),
\end{equation}
where $E_{m_j}=\mbox{diag}(\pm1,\pm1,\ldots,\pm1)$ is of order $m_j\times m_j$.
 And furthermore, 
$$\tilde{X}_j=X_jQ_jD_{jj},\ \ \ \tilde{X}_jJ_j=X_jQ_jD_{jj}J_j=X_jQ_jJ_jD_{jj}=X_jJ_jQ_jD_{jj}.$$

Case 2. $\lambda_j\in\mathbb{R}$, $j=r+1,\ldots, s$. And $J_j$ and $\Gamma_{jj}$ are defined by (\ref{xxg-4}) and (\ref{xxg-2-1}), respectively, $j=r+1,\ldots,s$. Since $\Gamma_{jj}$ is nonsingular and skew-symmetric, it follows that $n_j=2m_j$, $m_j\in\mathbb{Z}^+$. Then, there exist an unitary matrx $Q_{j}\in\mathbb{R}^{n_j\times n_j}$ such that 
$$Q_{j}^T\Gamma_{jj}Q_j=\begin{bmatrix}
0 & T_j\\
-T_j & 0\\
\end{bmatrix},$$
in which $T_j=\mbox{diag}(\tau^{(j)}_1,\ldots,\tau^{(j)}_{m_j})$, $\tau_{k}^{(j)}>0$, $k=1,\ldots,m_j$. Define $D_{jj}=\mbox{diag}(T_j,T_j)^{1/2}$. Since $N_j=0$, it is easy to verify that $J_jQ_j=Q_jJ_j$ and $J_jD_{jj}=D_{jj}J_j$ for $j=r+1,\ldots, s$. Then
\begin{equation}\label{xxg-6}
\tilde{\Gamma}_{jj}:=D^{-1}_jQ_j^T\Gamma_{jj}Q_jD_{j}^{-1}=\begin{bmatrix}
0 & I_{m_j}\\
-I_{m_j}& 0\\
\end{bmatrix},
\end{equation}
and 
$$\tilde{X}_{j}=X_jQ_jD_j^{-1},\ \ \tilde{X}_jJ_j=X_jQ_jD_j^{-1}=X_jJ_jQ_jD_j^{-1}.$$

Let 
$$\tilde{X}=[\tilde{X}_1,\tilde{X}_2,\ldots,\tilde{X}_s], \ \ \tilde{\Gamma}:=\mbox{diag}(\tilde{\Gamma}_{11},\tilde{\Gamma}_{22}\ldots,,\tilde{\Gamma}_{ss}).$$
$$Q=\mbox{diag}(Q_1,Q_2,\ldots,Q_s),\ \ D=\mbox{diag}(D_{11},D_{22},\ldots,D_{ss}).$$
We can see from (\ref{xxg-8}) and (\ref{xxg-6}) that 
$$\tilde{\Gamma}=D^{-1}Q^T\Gamma QD^{-1}=\left(DQ^TX_L^T\mathcal{A}X_LQD\right)^{-1}=(\tilde{X}_L^T\mathcal{A}\tilde{X}_L)^{-1},$$
has the structure specified in (\ref{xxg-1}), which means that $\tilde{\Gamma}$ is indeed the matrix corresponding to the standard pair $(\tilde{X},\ J)$, where $\tilde{X}=XQD$ and $\tilde{X}_L^T=[\tilde{X}^T, J^T\tilde{X}^T, \ldots, (J^{k-1})^T\tilde{X}^T]$. \end{proof}

\begin{theorem}\label{thm-T-o}
Suppose that all eigenvalues of $T$-odd $P(\lambda)$ are semi-simple with a real standard pair $(X, J)$ given by (\ref{gss-8-1}) and (\ref{gss-8-2}). Then there exists a real standard pair 
$(\tilde{X}, J)$ such that the corresponding matrix $\tilde{\Gamma}\in\mathcal{S}_{(J,T,-1,1)}$ has the structure
\begin{equation}\label{xxg-1-odd}
\begin{array}{ll}
\tilde{\Gamma}=&\mbox{diag}\left(\begin{bmatrix}
0&\mathfrak{I}_{2n_1}\\
\mathfrak{I}_{2n_1}&0\\
\end{bmatrix}, \cdots, \begin{bmatrix}
0&\mathfrak{I}_{2n_{r_1}}\\
\mathfrak{I}_{2n_{r_1}}&0\\
\end{bmatrix}, \begin{bmatrix}
E_{n_{r_1+1}} &0\\
0& E_{n_{r_1+1}}\\
\end{bmatrix}, \cdots, \begin{bmatrix}
E_{n_{r_2}} &0\\
0& E_{n_{r_2}}\\
\end{bmatrix},\right.\\
&\ \ \ \ \ \ \ \ \  \left. \begin{bmatrix}
0 & I_{n_{r_2+1}}\\
I_{n_{r_2+1}}  & 0\\
\end{bmatrix},\cdots,\begin{bmatrix}
0 & I_{n_{s-1}}\\
I_{n_{s-1}}  & 0\\
\end{bmatrix},E_{n_s}\right),\\
\end{array}
\end{equation} 
where  $\mathfrak{I}_{2n_k}=\mbox{diag}(I_{n_k}, -I_{n_k})$, $k=1,\ldots,r_1$, $E_{n_k}=\mbox{diag}(\pm1,\pm1,\ldots,\pm1)$ is of order $n_k\times n_k$ for $k=r_1+1,\ldots, r_2, s$. 
\end{theorem}

\begin{proof}
According to the matrix $\Gamma_{jj}$ given by (\ref{xzgs-5-1}), (\ref{xzgs-5-2}) and (\ref{xzgs-5-3}), we discuss the structure of $\Gamma_{jj}$ in four different cases. 

Case 1. $\lambda_j=\alpha_j+{\rm i}\beta_j$ with $\alpha_j, \beta_j>0$, and $J_j$ and $\Gamma_{jj}$ are given by (\ref{szgs-6-1}) and (\ref{xzgs-5-1}), respectively, for $j=1,\ldots,r_1.$ For the matrix $\tilde{U}_j=\bigl[\begin{smallmatrix}
\tilde{U}_{j1} & \tilde{U}_{j2}\\
\tilde{U}_{j2} & -\tilde{U}_{j1}\\
\end{smallmatrix}\bigr]$ given in (\ref{xzgs-5-1}), we can see from Lemma \ref{lem-8} that there exist two $2n_j\times 2n_j$ orthogonal-symplecitc matrices $Q_{j1}, Q_{j2}$ such that 
$$Q_{j1}^T\begin{bmatrix}
\tilde{U}_{j1} & \tilde{U}_{j2}\\
\tilde{U}_{j2} & -\tilde{U}_{j1}\\
\end{bmatrix}Q_{j2}=\begin{bmatrix}
D_j & 0\\
0 & -D_j\\
\end{bmatrix},$$
where $D_j=\mbox{diag}(a_1^{(j)},\ldots,a^{(j)}_{n_j})$, $a^{(j)}_k>0$, $k=1,\ldots,n_j$. Let $Q_j=\bigl[\begin{smallmatrix}
Q_{j1} & 0\\
0 & Q_{j2}\\
\end{smallmatrix}\bigr]$, $D_{jj}=\mbox{diag}(D_j, D_j, D_j, D_j)^{1/2}$ for $j=1,2,\ldots, r_1$. Then
$$\tilde{\Gamma}_{jj}:=D_{jj}^{-1}Q_j^T\Gamma_{jj}Q_jD_{jj}^{-1}=\begin{bmatrix}
0 & \mathfrak{I}_{2n_j}\\
\mathfrak{I}_{2n_j} & 0\\
\end{bmatrix},$$
where $\mathfrak{I}_{2n_j}=\mbox{diag}(I_{n_j}, -I_{n_j})$, and furthermore
$$\tilde{X}_j=X_jQ_jD_{jj},\ \ \tilde{X}_jJ_j=X_jQ_jD_{jj}J_j=X_jQ_jJ_jD_{jj}=X_jJ_jQ_jD_{jj}.$$

Case 2. $\lambda_j={\rm i}\beta_j$ with $\beta_j>0$ and $J_j$ and $\Gamma_{jj}$ are given by (\ref{szgs-6-2}) and (\ref{xzgs-5-2}), respectively, for $j=r_1+1,\ldots,r_2.$ Note that $\tilde{W}_{j1}^T=\tilde{W}_{j1}$ and $\tilde{W}_{j2}^T=-\tilde{W}_{j2}$, it follows that the $2n_j\times 2n_j$ matrix $\Gamma_{jj}=\bigl[\begin{smallmatrix}
\tilde{W}_{j1} & \tilde{W}_{j2}\\
-\tilde{W}_{j2} & \tilde{W}_{j1}\\
\end{smallmatrix}\bigr]$ is symmetric skew-Hamiltonian. By Lemma \ref{lem-7}, there exists an orthogonal-symplectic matrix $Q_j$ such that 
$$Q_j^T\Gamma_{jj}Q_j=\begin{bmatrix}
D_j & 0\\
0 & D_j\\
\end{bmatrix},$$
where $D_j=\mbox{diag}(b^{(j)}_1,\ldots, b^{(j)}_{n_j})$, $b^{(j)}_k\in\mathbb{R}$ for $k=1,\ldots, n_j$. Define $D_{jj}=\mbox{diag}(|D_j|, |D_j|)^{1/2}$. It is easy to verify that $J_jQ_j=Q_jJ_j$, and 
$$\tilde{\Gamma}_{jj}:=D_{jj}^{-1}Q_j^T\Gamma_{jj}Q_jD_{jj}^{-1}=\begin{bmatrix}
E_{n_j} & 0\\
0 & E_{n_j}\\
\end{bmatrix},$$
where $E_{n_j}=\mbox{diag}(\pm 1,\pm 1,\ldots,\pm 1)$ is of order $n_j\times n_j$. And, 
$$\tilde{X}:=X_jQ_jD_{jj},\ \ \tilde{X}_jJ_j=X_jQ_jD_{jj}J_j=X_jQ_jJ_jD_{jj}=X_jJ_jQ_jD_{jj}.$$

Case 3. $\lambda_j=\alpha_j$ with $\alpha_j>0$, and $J_j$ and $\Gamma_{jj}$ are given by (\ref{szgs-6-3}) and (\ref{xzgs-5-3}), respectively, for $j=r_2+1,\ldots,s-1.$ Let the SVD of $\tilde{V}_{j2}$ be 
$$\tilde{V}_{j2}=Q_{j1}D_jQ_{j2}^T,\ \ D_j=\mbox{diag}(c^{(j)}_1,\ldots,c^{(j)}_{n_j}), c^{(j)}_k>0, k=1,\ldots,n_j,$$
were $Q_{j1}, Q_{j2}\in\mathbb{R}^{n_j\times n_j}$ are two orthogonal matrices. Let $Q_j=\mbox{diag}(Q_{j1}, Q_{j2}))$ and define $D_{jj}=\mbox{diag}(D_j, D_j)^{1/2}$. Then we have 
$$\tilde{\Gamma}_{jj}:=D_{jj}^{-1}Q_j^T\Gamma_{jj}Q_jD_{jj}^{-1}=\begin{bmatrix}
0 & I_{n_j}\\
I_{n_j}& 0\\
\end{bmatrix},$$
and 
$$ \tilde{X}_j:=X_jQ_jD_{jj},\ \ \tilde{X}_jJ_j=X_jQ_jD_{jj}J_j=X_jQ_jJ_jD_{jj}=X_jJ_jQ_jD_{jj}.$$

Case 4. $\lambda_s=0$, $J_s=0$ and $\Gamma_{ss}$ is a real symmetric matrix of order $n_s\times n_s$. Clearly, there exists an orthogonal matrix $Q_s\in\mathbb{R}^{n_s\times n_s}$ such that 
$$Q_s^T\Gamma_{ss}Q_s=D_s,$$
where $D_s=\mbox{diag}(d_1,\ldots,d_{n_s})$, $d_k\in\mathbb{R}$, $k=1,\ldots,n_s$. Define $D_{ss}=|D_{s}|^{1/2}$. Then
$$\tilde{\Gamma}_{ss}:=D_{ss}^{-1}Q_s^T\Gamma_{ss}Q_sD_{ss}^{-1}=E_s,$$
where $E_s=\mbox{diag}(\pm 1,\ldots,\pm 1)$ is of order $n_s\times n_s$, and 
$$\tilde{X}_s:=X_sQ_sD_{ss},\ \ \tilde{X}_sJ_s=X_sQ_sD_{ss}J_s=X_sQ_sJ_sD_{ss}=X_sJ_sQ_sD_{ss}.$$

Now, let 
$$\tilde{X}=[\tilde{X}_1,\tilde{X}_2,\ldots,\tilde{X}_s], \ \ \tilde{\Gamma}:=\mbox{diag}(\tilde{\Gamma}_{11},\tilde{\Gamma}_{22}\ldots,,\tilde{\Gamma}_{ss}).$$
$$Q=\mbox{diag}(Q_1,Q_2,\ldots,Q_s),\ \ D=\mbox{diag}(D_{11},D_{22},\ldots,D_{ss}).$$
Observe that the matrix 
$$\tilde{\Gamma}=D^{-1}Q^T\Gamma QD^{-1}=\left(DQ^TX_L^T\mathcal{A}X_LQD\right)^{-1}=(\tilde{X}_L^T\mathcal{A}\tilde{X}_L)^{-1},$$
has the structures specified in (\ref{xxg-1-odd}), which means that $\tilde{\Gamma}$ is indeed the matrix corresponding to the standard pair $(\tilde{X}, J)$, where $\tilde{X}=XQD$ and $\tilde{X}_L^T=[\tilde{X}^T, J^T\tilde{X}^T, \ldots, (J^{k-1})^T\tilde{X}^T]$.
\end{proof}

\begin{theorem}\label{them-h}
Suppose that all eigenvalues of the $(H,\epsilon_1, 1)$-structured matrix polynomial $P(\lambda)$ are semi-simple with a standard pair $(X,\ J)$ given by (\ref{gss-8-1}) and (\ref{gss-8-2}) . Then there exists a  standard pair 
$(\tilde{X}, J)$ such that the corresponding matrix $\tilde{\Gamma}\in\mathcal{S}_{(J,H,\epsilon_1,1)}$ has the structure
\begin{equation}\label{xxxg-1}
\tilde{\Gamma}=\mbox{diag}\left(\begin{bmatrix}
0 & I_{n_1}\\
I_{n_1} & 0\\
\end{bmatrix},\cdots,\begin{bmatrix}
0 & I_{n_r}\\
I_{n_r} & 0\\
\end{bmatrix}, E_{n_{r+1}},\cdots,E_{n_s}\right),
\end{equation} 
where  $E_{n_j}=\mbox{diag}(\pm1,\pm1,\ldots,\pm1)$ is of order $n_j\times n_j$, $j=r+1,\ldots,s.$
\end{theorem}
\begin{proof}
We only give the proof of $H$-Hermitian (i.e., $\epsilon_1=1$), and the case of $H$-odd (i.e., $\epsilon_1=-1$) can be proved similarly.  Since $(X,\ J)$ is a standard pair of $P(\lambda)$, it follows from Theorem \ref{thm-4} that $P(\lambda)$ has a spectral decomposition as in (\ref{gs-7}), in which the matrix $\Gamma\in\mathcal{S}_{(J, H, 1, 1)}$ has the form 
\begin{equation}\label{xxg-9}
\Gamma=\mbox{diag}\left(\begin{bmatrix}
0 & W_1\\
W_1^H & 0\\
\end{bmatrix},\cdots, \begin{bmatrix}
0 & W_r\\
W_r^H & 0\\
\end{bmatrix}, W_{r+1},\ldots, W_s\right),
\end{equation}
where $W_j\in\mathbb{C}^{n_j\times n_j}$, $j=1,\ldots, s$, and $W_j^H=W_j$ for $j=r+1,\ldots,s$. Now, we give the structure of $W_j$ according to $J_j$.

{\rm (1)} $\lambda_j\in\mathbb{C}/\mathbb{R}$, $j=1,\ldots, r$. Let the SVD of $W_j$ be  
$$W_j=\tilde{U}_jD_j \tilde{V}_j^H, D_j=\mbox{diag}(\eta^{(j)}_1,\ldots,\eta^{(j)}_{n_j}), \eta^{(j)}_k>0, k=1,\ldots,n_j,$$
where $\tilde{U}_j, \tilde{V}_j\in\mathbb{C}^{n_j\times n_j}$ are unitary matrices. Let $P_j=\mbox{diag}(\tilde{U}, \tilde{V})$ and $D_{jj}=\mbox{diag}(D_j,D_j)^{1/2}$. Then
\begin{equation}\label{xxg-10}
\tilde{\Gamma}_{jj}:=D_{jj}^{-1}P_j\begin{bmatrix}
0 & W_j\\
W_j^H & 0\\
\end{bmatrix}P_j^HD_{jj}^{-1}=\begin{bmatrix}
0 & I_{n_1}\\
I_{n_1} & 0\\
\end{bmatrix},
\end{equation}
and, furthermore,
$$\tilde{X}_j=X_jP_j^HD_{jj}, \ \ \tilde{X}_jJ_j=X_jP_j^HD_{jj}J_j=X_jP_j^HJ_jD_{jj}=X_jJ_jP_j^HD_{jj},$$
since $N_j=0$ by assumption. 

{\rm (2)} $\lambda_j\in\mathbb{R}$, $j=r+1,\ldots, s$. Since $W_j^H=W_j$ is nonsingular, there exists an unitary matrix $P_j\in\mathbb{C}^{n_j\times n_j}$ such that $P_jW_jP_j^H=D_j$, where $D_j=\mbox{diag}(\xi^{(j)}_1,\ldots,\xi^{(j)}_{n_j})$ is nonsingular with $\xi^{(j)}_k\in\mathbb{R}$, $k=1,\ldots,n_j$. Let $D_{jj}=|D_j|^{1/2}$. Then, we have 
\begin{equation}\label{xxg-11}
\tilde{\Gamma}_{jj}=D_{jj}^{-1}P_jW_jP_j^HD_{jj}^{-1}=E_j,
\end{equation}
where $E_j=\mbox{diag}(\pm1,\ldots,\pm1)$ is of order $n_j\times n_j$, $j=r+1,\ldots, s$, and 
$$\tilde{X}_j=X_jP_j^HD_{jj}, \ \ \tilde{X}_jJ_j=X_jP_j^HD_{jj}J_j=X_jP_j^HJ_jD_{jj}=X_jJ_jP_j^HD_{jj}.$$

Now define 
$$P=\mbox{diag}(P_1,P_2,\ldots,P_s), D=\mbox{diag}(D_{11},D_{22},\ldots,D_{ss}).$$
Observe that the matrix
$$\tilde{\Gamma}:=D^{-1}P\Gamma P^HD^{-1}=\mbox{diag}(\tilde{\Gamma}_{11},\tilde{\Gamma}_{22},\ldots,\tilde{\Gamma}_{ss})$$
has the structure specified in (\ref{xxxg-1}), and we have 
$$\tilde{X}=XP^HD, \ \ \tilde{X}J=XP^HDJ=XP^HJD=XJP^HD,$$
therefore, $\tilde{\Gamma}$ is indeed the matrix corresponding to the standard pair $(\tilde{X}, J)$. 
\end{proof}

Similar to the proof of Theorem \ref{them-h}, we can prove the following result. 

\begin{theorem}\label{thm-H-s}
Suppose that all eigenvalues of the $(H,\epsilon_1, -1)$-structured matrix polynomial $P(\lambda)$ are semi-simple with a standard pair $(X, J)$ given by (\ref{gss-8-1}) and (\ref{gss-8-2}) . Then there exists a  standard pair 
$(\tilde{X}, J)$ such that the corresponding matrix $\tilde{\Gamma}\in\mathcal{S}_{(J,H,\epsilon_1,-1)}$ has the structure
\begin{equation}\label{xxxg-1}
\tilde{\Gamma}=\mbox{diag}\left(\begin{bmatrix}
0 & I_{n_1}\\
-I_{n_1} & 0\\
\end{bmatrix},\cdots,\begin{bmatrix}
0 & I_{n_r}\\
-I_{n_r} & 0\\
\end{bmatrix}, E_{n_{r+1}},\cdots,E_{n_s}\right),
\end{equation} 
where  $E_{n_j}=\mbox{diag}(\pm{\rm i},\pm{\rm i},\ldots,\pm{\rm i})$ is of order $n_j\times n_j$, $j=r+1,\ldots,s.$
\end{theorem}

\section{Conclusions}

In this paper, we have derived the spectral decomposition of the $(\star, \epsilon_1, \epsilon_2)$-structured matrix polynomials $P(\lambda)$ in the unified form, in which the parameter matrix $\Gamma$ plays an important role. 
With a standard pair $(X,\ J)$ of $P(\lambda)$, the parameter matrix $\Gamma$ possesses a block diagonal structure and $\star$-symmetric upper triangular Hankel/skew-Hankel blocks forms as in Theorem \ref{thm-4}. Using the recursive relationship between the coefficient matrices of $P(\lambda)$, we have provided equivalent expressions of these coefficient matrices. In the special case that all the eigenvalues of $P(\lambda)$ are semi-simple, the matrix $\Gamma\in\mathcal{S}_{(J,\star,\epsilon_1,\epsilon_2)}$ has some special forms 
as shown in Theorem \ref{thm-T-s}, Theorem \ref{thm-T-o}, Theorem \ref{them-h} and Theorem \ref{thm-H-s}.






\end{document}